\documentclass[oneside]{article}

\usepackage[accepted]{Arxiv_style}
\usepackage[round]{natbib}

\usepackage{amssymb} 
\usepackage{mathtools}
\usepackage[caption=false, font=footnotesize]{subfig}
\usepackage{amsthm}
\usepackage{xcolor}
\usepackage{hyperref}

\newtheorem{lemma}{Lemma}
\newtheorem{definition}{Definition}
\newtheorem{theorem}{Theorem}
\newtheorem{corollary}{Corollary}
\newtheorem{remark}{Remark}
\newtheorem{assumption}{Assumption}

\newcommand{\dimMatCom}[0]{K}
\newcommand{\dimMatF}[0]{n}
\newcommand{\dimMatC}[0]{m} 
\newcommand{\matF}[0]{\ensuremath{W}}
\newcommand{\matC}[0]{\ensuremath{\tilde{W}}}
\newcommand{\matLinSens}[0]{\ensuremath{B}}

\newcommand{\comFine}[0]{\ensuremath{P}}

\DeclareMathOperator*{\argmax}{arg\,max}

\newcommand{\genSyncMat}[0]{\ensuremath{\Phi}}

\newcommand{\genSyncVec}[0]{\ensuremath{\mathbf{\phi}}}

\newcommand{\dimSum}[0]{\ensuremath{J}}
\newcommand{\syncComAs}[0]{\ensuremath{\mathbf{\genSyncMat}}}
\newcommand{\fullComAs}[0]{\ensuremath{\mathbf{\theta}}}

\newcommand{\protoCovSize}[0]{\ensuremath{r}}

\newcommand{\biMatC}[0]{\ensuremath{\matC^{(\textrm{b})}}}
\newcommand{\genComConnect}[0]{\ensuremath{U}}
\newcommand{\nonSyncDimMatCom}[1]{\ensuremath{\dimMatCom^{(#1)}}} 
\newcommand{\thresholdMatCin}[0]{\ensuremath{\tilde{\tau}}}

\newcommand{\COvar}[0]{\ensuremath{\nu}}

\newcommand{\intraComProb}[0]{\ensuremath{p}}
\newcommand{\extraComProb}[0]{\ensuremath{q}}
\newcommand{\constintraComProb}[0]{\ensuremath{\alpha}}
\newcommand{\constextraComProb}[0]{\ensuremath{\beta}}
\newcommand{\scaleFineComProb}[0]{\ensuremath{f(\dimMatF)}}
\newcommand{\upConstScaleFineComProb}[0]{\ensuremath{f_0}}
\newcommand{\scaleCoarseComProb}[0]{\ensuremath{\tilde{f}(\dimMatC,\dimMatF,\protoCovSize)}}
\newcommand{\scaleCoarseComProbAt}[2]{\ensuremath{\tilde{f}(\dimMatC,{#1},{#2})}}
\newcommand{\scaleComProbAt}[1]{\ensuremath{f(#1)}}
\newcommand{\scaleComProbAtPower}[2]{\ensuremath{f^{#2}(#1)}}
\newcommand{\sbmDegree}[0]{\ensuremath{d}}
\newcommand{\scaledCHdiv}[0]{\ensuremath{D}}
\newcommand{\CHdiv}[0]{\ensuremath{D_{+}}}
\newcommand{\genConstDegScale}[0]{\ensuremath{c_0}}
\newcommand{\profileSet}[0]{\ensuremath{\Upsilon^{(\COvar)}}}
\newcommand{\strictGreaterBigO}[0]{\ensuremath{\Omega}}
\begin{document}

%
\runningtitle{Graph Community Detection from Coarse Measurements}

%

\twocolumn[

\aistatstitle{Graph Community Detection from Coarse Measurements: Recovery Conditions for the Coarsened Weighted Stochastic Block Model}
 
\aistatsauthor{ Nafiseh Ghoroghchian \And Gautam Dasarathy  \And  Stark C. Draper }
\aistatsaddress{University of Toronto \\  Vector Institute  \And  Arizona State University \And University of Toronto   } 
]

\begin{abstract}
We study the problem of community recovery from coarse measurements of a graph. In contrast to the problem of community recovery of a fully observed graph, one often encounters situations when measurements of a graph are made at low-resolution, each measurement integrating across multiple graph nodes. Such low-resolution measurements effectively induce a coarse graph with its own communities. Our objective is to develop conditions on the graph structure, the quantity, and properties of measurements, under which we can recover the community organization in this coarse graph. In this paper, we build on the stochastic block model by mathematically formalizing the coarsening process, and characterizing its impact on the community members and connections. Through this novel setup and modeling, we characterize an error bound for community recovery. The error bound yields simple and closed-form asymptotic conditions to achieve the perfect recovery of the coarse graph communities.  
\end{abstract}

\section{Introduction}
Community detection (a.k.a. clustering) in a graph is the problem of identifying groups of nodes with similar behaviour  \citep{fortunato2016community,von2007tutorial,abbe2017community}. Identifying communities is usually the first analysis tool used to draw an initial observation from data \citep{yang2013overlapping}.
A community in a graph refers to a group of nodes that are more similar to each other than to the rest of the graph.
The notion of similarity most conventionally means assortativity, i.e. denser intra-community links in an unweighted graph where no weight or label is associated with the graph edges \citep{fortunato2010community}. However, the group similarity notion has been extended to other forms of connectivity, as well as to weighted networks \citep{fortunato2016community}.
Cluster formation is proven to be a universal structure in real networks \citep{yang2015defining}.
As a result, detecting communities in networks has become a central question to a great body of prediction and inference tasks, with applications in 
network neuroscience \citep{sporns2016modular, bassett2017network,betzel2019community}, social networks \citep{yang2013overlapping}, collaboration networks \citep{hou2008structure}, and biological networks \citep{girvan2002community}.

While existing methods for community detection have been effective in modeling, studying, and recovering communities from finely detailed, high-resolution graphs \citep{fortunato2016community}, 
there are various scenarios where a large-scale graph is not fully observable and should be coarsened
due to restrictions imposed by the measuring instrument (will be exemplified shortly) \citep{betzel2017multi}, limitations of the storage memory, high sampling costs, computational tractability \citep{dabagiacomparing,serrano2009extracting}, restricted accessibility to data, and the creation of multi-scale representations for graphs \citep{safro2015advanced,loukas2019graph}. 
Discovering the latent community structure from the coarse measured graph is a valuable objective of many graph-based tasks \citep{mucha2010community,betzel2019community}.

\begin{figure*}[!htp]
\centerline{\includegraphics[width=17.4cm]{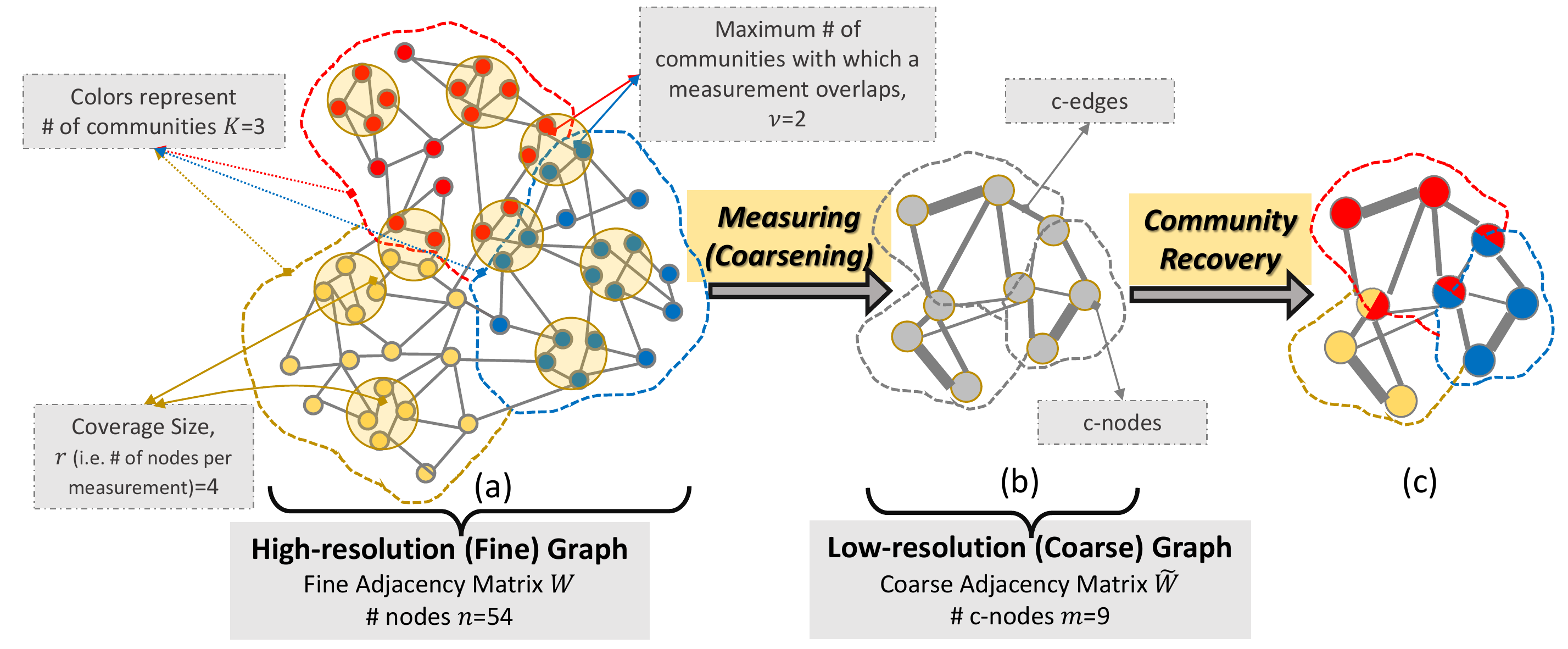}}\label{fig:visual_coarsening_communities}
\caption{\small Visual illustration of (a) the underlying high-resolution (fine) graph, (b) the measurement (coarsening) procedure whose result is modeled as a coarse graph, and (c) the effect of the coarsening on the community structure, whose recovery is the objective of this paper. Some notations used in this paper, with their values realized for this figure, are annotated.}
\end{figure*}

Although conventional community detection models can be directly applied to the coarse measured graphs \citep{betzel2017multi}, a fundamental understanding of the impact of coarsening on the community structure and recovery is missing. 
Fig.~\ref{fig:visual_coarsening_communities} illustrates how the coarse measurement process can obscure the high-resolution graph structure. The figure shows that as coarsening reduces the size of the graph, introduces heterogeneity in the edge weights, which can potentially cause a drift away from the true community structure. 

The study of clustering from coarse measured graphs enables the characterization of contributing factors to their community recovery. Such characterization leads to identifying the barriers in community detection from a coarse graph, which can potentially improve the clustering by applying adjustments to the measurement and community recovery process.
Such clustering characterization and recovery improvement are crucial to many fields including neuroscience. Often in the study of the brain on a large scale, the scientific measuring instruments are quite coarse and cannot directly monitor the activity of all the neurons in the brain, which is as high as $14$ billion. Hence, one is restricted to collect aggregate signals from bundles of neurons \citep{osorio2016epilepsy,ghoroghchian2020node}, from which a low-resolution functional brain graph is generated \citep{friston2011functional,ghoroghchian2018hierarchical}. The communities identified in the measured graph have been connected to brain cognitive and behavioral units, and they provide biomarkers for neurological diseases \citep{sporns2016modular,bassett2017network, lynn2019physics, patankar2020path}.


\textbf{Contributions}: 
In this paper, we study the community detection from coarse measured graphs, which to the best of our knowledge is the first analysis of this problem:
\begin{itemize}
    \item A random generative model is introduced for the coarse measured networks. A mathematical framework is defined that characterizes the measurement process, the coarse graph, as well as the relationship between the community structure of the fine and coarse graphs.
    \item Simple and closed-form asymptotic conditions are developed on the graph structure, the quantity and properties of the measurements,  under which the community organization of the coarse graph is recovered. The recovery error is characterized, which facilitated studying the effects of various measurement- and structure-related parameters, who take part in improving or exacerbating the quality of the recovery. 
    \item Simulations are provided to compare the derived theoretical error bound with the performance of state-of-the-art community detection methods.
\end{itemize}
\textbf{Related Work}: 
{
While the problem of coarsening a \textit{known graph} has received considerable attention in the past \citep{karypis1998fast,harel2001clustering,kushnir2006fast,safro2015advanced,loukas2019graph,rahmani2020scalable}, to the best of our knowledge, this paper is the first to consider learning community structure from coarse summaries of an \textit{unknown graph}.
}

This paper is built upon the stochastic block model (SBM), a random generative model that is widely used as a canonical model in community detection literature \citep{abbe2017community}.
Although there are other approaches to detect communities, mainly based on modularity maximization and statistical inference \citep{fortunato2016community,javed2018community}, there are advantages to SBM that fit it to our purposes. SBM
provides a rich benchmark that facilitates its generalization to numerous variants  \citep{abbe2017community,fortunato2016community,funke2019stochastic}. Furthermore, the generative nature of SBMs allows for characterizing communities and their recovery \citep{abbe2017community}, which particularly serves the improvement of community detection. 
{
We start by using the vanilla symmetric SBM to model the fine scale graph, which we consider a \emph{latent model} that underlies the observed coarse graph. Under this model, we show that the coarse graph becomes a weighted and mixed membership (or overlapping) variant of the SBM. 
}

{
The mixed membership SBM (MMSBM) is another relevant paradigm to our purposes and could serve as a good model when directly applied at the measurement (coarsened) level. However in the current paper, we start with a model of the fine graph and characterize the  coarse model as a function of the coarsening/measurement procedure. This is more natural given our goal is to infer community information about the underlying fine graph. Relatedly, as far as we know, most papers on MMSBM such as \citep{dulac2020mixed} are algorithmic-oriented and do not contain theoretical analysis of the community recovery performance similar to our work in this paper. Few existing works that include theoretical analysis \citep{mao2017mixed}, do not model weighted edges and do not focus on coarsening, the two components that are crucial to our setup.
}

\section{Model}\label{sec:model}
Consider an unweighted graph $G=G(\mathcal{V},E)$, where $\mathcal{V}$ is a set of nodes of cardinality $|\mathcal{V}|=\dimMatF$, and $E$ is a set of pairs of nodes, referred to as edges. Alternative to $E$ and since the graph is unweighted, we can represent the edges using an adjacency matrix $\matF\in\lbrace 0,1\rbrace^{\dimMatF\times \dimMatF}$, where each node of the graph is labeled by a unique number in the index set $[\dimMatF]\triangleq \lbrace 1,2,\cdots,\dimMatF \rbrace$, and $\matF_{uv}=1$ shows the existence of an edge between nodes $u$ and $v$. 

We assume an underlying community structure on $\mathcal{V}$, which partitions the node set into disjointed sets $\mathcal{V}=\displaystyle\cup_{k=1}^{\dimMatCom}\mathcal{V}_k$. For all $k\in[\dimMatCom]$, $\mathcal{V}_k$ represents the set of the nodes that belong to community $k$. Each node belongs to only one of the $\dimMatCom$ communities.
The intra-connection among nodes in the same community is different from their connection to the rest of the graph. 
Let $\comFine\in\lbrace 0,1\rbrace^{\dimMatCom\times \dimMatF}$ be the true community assignment matrix, where $\comFine_{ku}=1$ iff node $u$ belongs to community $k$, i.e. 
\begin{align}
   \comFine_{ku}=\left\lbrace
   \begin{array}{ll}
        1 & \textrm{if } u\in\mathcal{V}_k  \\
        0 & \textrm{else}
   \end{array}
   \right..
\end{align}

A graph is drawn under the Symmetric Stochastic Block Model (SSBM) characterised by $\intraComProb$ and $\extraComProb$, where the probability of having an edge between two nodes is independently distributed according to $\textrm{Bernoulli}(\intraComProb)$, for two nodes in the same community, and $\textrm{Bernoulli}(\extraComProb)$ for nodes in different communities. Also, the nodes are assigned to communities in a uniform and independent manner. We let $\matF$ be distributed according to  $\matF\sim \textrm{SSBM}(\dimMatF,\dimMatCom,\intraComProb,\extraComProb)$ conditional on $\comFine$, i.e.,
\begin{align}\label{SSBM_dist}
    \matF_{uv}\sim\left\lbrace \begin{array}{cc}
        \textrm{Bernoulli}(\intraComProb) & \textrm{if } \exists k\in[\dimMatCom]: \comFine_{ku}=1, \comFine_{kv}=1 \\
        \textrm{Bernoulli}(\extraComProb) & \textrm{else}
    \end{array}\right.,
\end{align} 
We assume a general scaling behaviour for $\intraComProb,\extraComProb$ by defining the constants $0<\alpha,\beta<\infty$ and a scaling factor $\scaleFineComProb$, where:
\begin{align}\label{comConnect_constant_with_scaling}
    \intraComProb\triangleq \constintraComProb\scaleFineComProb, \quad \extraComProb\triangleq \constextraComProb\scaleFineComProb.
\end{align}
$\scaleFineComProb$ tracks the changes in the graph density as a function of the graph size. As $\dimMatF$ increases, $\scaleFineComProb$ may remain unchanged, or it may get smaller, i.e. the graph becomes sparser as it grows. The latter sparsity assumption has been considered in existing literature, as it fits to many real-world applications, including biological, social, and collaborative networks \citep{abbe2017community,mossel2014consistency,abbe2015exact,abbe2015community}.

In real applications, $G$ can be very large, in the order of millions or billions of nodes.  In general, the population is much larger than the number of communities (e.g., there are many more citizens than cities) and so $\dimMatF\gg\dimMatCom$. We often cannot observe the existence (or lack of existence) of all $\frac{\dimMatF(\dimMatF-1)}{2}$ possible connections and instead measure $\dimSum$ \textit{summaries} of associations. 

One possible choice to collect a simplified and interpretable set of  summary measurements (more explanations come shortly), is to define a set of \textit{disjointed} measurement vectors $\lbrace \mathbf{b}_1,\mathbf{b}_2, \cdots, \mathbf{b}_{\dimMatC}\rbrace$, all satisfying $\mathbf{b}_i\in\lbrace 0,1\rbrace^{\dimMatF}$ and $\mathbf{b}_i\mathbf{b}_j^\intercal =0$ for all $i\in[\dimMatC]$ different from $j\in[\dimMatC]$. The latter condition means measurement vectors do not overlap, i.e. each node is measured at most one time. Each summary, denoted by $s_{\ell}$ for $\ell\in[\dimMatC^2]$, is defined as:
\begin{align}\label{summary_measurements}
\begin{array}{rl}
     s_{\ell} & =\displaystyle\sum_{u\in \textrm{supp}(\mathbf{b}_{\lceil \ell/\dimMatC\rceil})} \displaystyle\sum_{v\in \textrm{supp}(\mathbf{b}_{\ell \textrm{ mod } \dimMatC})} \matF_{uv} \\
     & = \mathbf{b}_{\lceil \ell/\dimMatC\rceil}\matF\mathbf{b}_{\ell \textrm{ mod } \dimMatC}^\intercal.
\end{array}
\end{align}
$\textrm{supp}(\mathbf{b}_{i})$ denotes the support of $\mathbf{b}_{i}$ and $|\textrm{supp}(\mathbf{b}_{i})|$ is the cardinality of the support.
Equation \eqref{summary_measurements} corresponds to the set of summary measurements one would get if one defines an $\dimMatC\times\dimMatF$ matrix whose rows are $\mathbf{b}_1,\mathbf{b}_2, \cdots, \mathbf{b}_{\dimMatC}$, and then collects $\dimMatC^2$ non-distinct (or $\frac{\dimMatC(\dimMatC+1)}{2}$ distinct) measurements as in \eqref{summary_measurements}, forming the following matrix equality:
\begin{align}\label{linear_coarsening_model}
    {\matC} = \matLinSens {\matF} \matLinSens^\intercal.
\end{align}  
Such measurement model 
is a natural choice in existing applications.
For instance, linear measurement of a high-dimensional signal appear in compressed sensing \citep{donoho2006compressed,draper2009compressed} which is further applied to Electroencephalogram (EEG) signal processing \citep{aviyente2007compressed} and image processing \citep{baraniuk2007compressive}, as well as in Covariance sketching \citep{dasarathy2015sketching}. For such linear measurements, the original and the measured graphs respectively model the Covariance (here, thresholded for weighted graphs) matrices of the original and the linearly measured signals. The measurement model in \eqref{linear_coarsening_model} is also a popular graph reduction method, where $\matC$ approximates $\matF$ by preserving some of its spectral properties \citep{safro2015advanced, loukas2019graph,jin2020graph}.  

Matrix $\matC$ can be thought of as the weighted adjacency matrix representation of a measured weighted graph $\tilde{G}\in\tilde{\mathcal{G}}(\tilde{\mathcal{V}},\tilde{E})$, where
$\tilde{\mathcal{V}}$ is the set of \textit{c-nodes}
\footnote{{
``c-'' stands for compound or coarse.
}} 
and $|\tilde{\mathcal{V}}|=\dimMatC$. 
$\tilde{E}$ is the set of \textit{c-edges}, consisting of pairs of c-nodes and a weight, i.e., $(i,j,\tilde{w})$.
Note that $\matC_{ij}$'s for all $i>j$ are independent random variables if the $\mathbf{b}_i$'s are disjoint. We return to this point, and the formal statistics of $\matC$, shortly.

\begin{definition}\label{def_homogenous}
A measurement matrix $\matLinSens$ is ``$\protoCovSize$-homogeneous'' if for all  $i\in[\dimMatC]$ there is a constant positive integer $\protoCovSize\leq \frac{\dimMatF}{\dimMatC}$ such that  $|\textrm{supp}(\mathbf{b}_i)|=\protoCovSize$.
\end{definition}
We assume the number of measured fine nodes that represent a c-node is the same for all c-nodes. We refer to this number as the \textit{coverage size} and denote it by $\protoCovSize$. Accordingly, the support of the rows of a homogeneous measurement matrix has cardinality equal to $\protoCovSize$.
We define the \textit{c-node profile} matrix:
\begin{align}\label{gen_profile_def}
    \genSyncMat \triangleq \matLinSens\comFine^\intercal,
\end{align}
whose dimension is $\dimMatC\times\dimMatCom$ and connects the measurement matrix $\matLinSens$ to the graph of community assignment matrix $\comFine$. $\genSyncMat$ displays the impact of coarsening on the community memberships. A c-node can belong to one community or multiple communities. Each row of the c-node profile matrix, $\genSyncVec_i$, is a length-$\dimMatCom$ vector that counts the number of nodes in each community in $G$ that is measured by the $i$-th c-node.
For instance, $\genSyncMat=\left[\begin{array}{ccc}
     0 & 4 & 0 \\
     2 & 0 & 2 \\
     4 & 0 & 0 \\
     0 & 2 & 2 
\end{array}\right]$ means that, all the $4$ fine nodes that map to the first (resp. the third) c-node belong to community $2$ (resp. $1$), while half of the $4$ fine nodes mapping to the second c-node belong to community $1$ and the other half belong to community $3$. 

The following Lemma derives the statistics of $\matC$.
\begin{lemma}\label{lemma:coarse_mat_stat}
Let $\matF\sim \textrm{SSBM}(\dimMatF,\dimMatCom,\intraComProb,\extraComProb)$ from which $\matC$ in \eqref{linear_coarsening_model} is measured under the $\protoCovSize$-homogeneous measurement assumption defined in Def.~\ref{def_homogenous}. Then $\matC_{ij}$'s are i.i.d. random variables for all $i>j$, with distribution:
\begin{align}\label{homogeneous_general_dist_matC}
    \begin{array}{ll}
       \matC_{ij} \sim & \textrm{PoissonBinomial}(\lbrace \intraComProb\rbrace^{\genSyncVec_i^\intercal\genSyncVec_j}, \lbrace \extraComProb \rbrace^{\protoCovSize^2-\genSyncVec_i^\intercal\genSyncVec_j})
    \end{array},
\end{align}
where the $\textrm{PoissonBinomial}$ in \eqref{general_dist_matC}, is a compact notation for a Poisson Binomial distribution with success probabilities $\genSyncVec_i^\intercal\genSyncVec_j$ of $p$'s and $\protoCovSize^2-\genSyncVec_i^\intercal\genSyncVec_j$ of $q$'s.
\end{lemma}
The proof is elaborated in Sec.~\ref{appendix:proof_lemma_coarse_mat_stat}  of the supplementary materials.
 
Each c-node can measure from members of one or multiple communities. We denote the maximum number of communities that overlap with a c-node, by $\COvar$, where $1\leq\COvar\leq \dimMatCom$. This is considered as a Community Overlap (CO) constraint, and is illustrated in the next Definition. 
\begin{definition}\label{def:PS_def}
A measurement matrix $\matLinSens$ is CO-$\COvar$ with respect to a graph $G\in\mathcal{G}(\mathcal{V},E)$ with community assignment matrix $\comFine$, if the profile matrix $\genSyncMat = \matLinSens\comFine^\intercal$ satisfies: $1 \leq |\textrm{supp}( \genSyncVec_i)|\leq \COvar \quad \forall i\in[\dimMatC]$.
\end{definition}
Def.~\ref{def:PS_def} means that the support of each row of $\matLinSens$ corresponds to at most $\COvar$ of the communities in $G$. 
The next definition is the last step to formalizing the coarse graph community structure.
\begin{definition}\label{def:balanced}
A measurement matrix $\matLinSens$ is ``balanced'' with respect to a graph $G\in\mathcal{G}(\mathcal{V},E)$ with community assignment matrix $\comFine$, if the profile matrix $\genSyncMat = \matLinSens\comFine^\intercal$  satisfies $\genSyncMat_{ik} = \genSyncMat_{ik'}$ for all $i\in[\dimMatC]$ and $k,k'\in \textrm{supp}(\genSyncVec_i)$.
\end{definition}
In other words, in a balanced-measured graph, an identical number of nodes are measured from each community.

The objective of this paper is to recover the c-node profile matrix $\genSyncMat$ from the measured graph $\matC$ in \eqref{linear_coarsening_model}. 
Let a maximum a posteriori ($\text{MAP}$) estimator take a measured graph $\tilde{G}$ with the true c-node profile matrix $\genSyncMat$, and returns its estimate $\hat{\genSyncMat}$ that assigns every c-node in $\tilde{G}$ to communities. We characterize an upper bound on the failure probability of the $\text{MAP}$ estimator. The error refers to assigning a wrong profile to at least one c-node, up to equivalent relabelling of communities. We also study the asymptotic conditions such that this error tends to zero. 

Recovering $\genSyncMat$ in \eqref{gen_profile_def} from the measured matrix $\matC$, without imposing additional constraints on $\genSyncMat$, is generally a very hard problem. Hence, we relax the problem to achieve tractability, by putting the constraints in Def. \ref{def_homogenous}, \ref{def:PS_def}, and \ref{def:balanced} on the measurement matrix (i.e. $\matLinSens$), with respect to the community assignment matrix (i.e. $\comFine$) of the graph. 
{
In many practical settings, assumptions such as homogeneity are reasonable. For instance, Electroencephalography (ECoG) signals are acquired from different brain regions using electrodes whose contact surface areas are the same.
Nevertheless, a relaxation of these assumptions is of considerable interest and will serve as a compelling avenue for future exploration.
}

In the next section, we state and study the community recovery problem under the CO-$\COvar$ constraint.

\section{Recovery under the Community Overlap (CO)-$\COvar$ constraint} \label{sec:general_sync}
In this section, 
we derive an upper bound on the $\text{MAP}$ recovery error of the profile matrix ${\genSyncMat}$, as described at the end of Sec.~\ref{sec:model}. The recovered profile matrix $\hat{\genSyncMat}$, estimates at most $\COvar$ communities from which each c-node measures.

\subsection{Main Results} \label{subsec:mainResultsgeneral_sync}
We begin sketching our main result by defining 
\begin{align}\label{size_extended_coms}
\nonSyncDimMatCom{\COvar} = \displaystyle\sum_{\ell=1}^{\COvar}{\dimMatCom\choose \ell},
\end{align}
the profile set:
\begin{align}\label{profile_set}
\begin{array}{l}
\profileSet \triangleq \lbrace  \genSyncVec_i| \genSyncVec_i = \mathbf{b}_i\comFine^\intercal, 1 \leq |\textrm{supp}(\genSyncVec_i)|\leq \COvar,
\\
 \qquad\qquad\quad\forall k\in \textrm{supp}(\genSyncVec_i): \genSyncVec_{ik} = \frac{\protoCovSize}{|\textrm{supp}(\genSyncVec_i)|} 
    \rbrace,
\end{array}
\end{align}
and a one-to-one function $h:\profileSet \xrightarrow[]{}[\nonSyncDimMatCom{\COvar}]$. Function $h$ maps a c-node profile to an \textit{extended} community indexed by $[\nonSyncDimMatCom{\COvar}]$ (more explanations in Sec.~\ref{subsec:nonSync_proof_techniques}). 
A probability matrix $U\in[0,1]^{\nonSyncDimMatCom{\COvar}\times\nonSyncDimMatCom{\COvar}}$ is defined, for all  $\mathbf{a},\mathbf{a}' \in \profileSet$ and $k=h(\mathbf{a}),k'=h(\mathbf{a}')$, as:
\begin{align}\label{U_poissonBinomial_initDef}
    \begin{array}{ll}
       \genComConnect_{k,k'}   & =
\mathbb{P}(X\geq \protoCovSize^2(\thresholdMatCin \intraComProb + (1-\thresholdMatCin)\extraComProb))
    \end{array},
\end{align}
where $0\leq\thresholdMatCin\leq 1$ and $X$ is an auxiliary random variable distributed as:
\begin{align}\label{X_auxiliaryRV_PoissBin_def}
    \begin{array}{ll}
       X \sim & \textrm{PoissonBinomial}(\lbrace \intraComProb\rbrace^{\mathbf{a}^\intercal\mathbf{a}'}, \lbrace \extraComProb \rbrace^{\protoCovSize^2-\mathbf{a}^\intercal\mathbf{a}'})
    \end{array}.
\end{align}
Sec.~\ref{subsec:nonSync_proof_techniques} will elaborate on the reasons behind these definitions, using the \textit{binarization} of the coarse measured graph, i.e. mapping the c-edge weights to zero or one. Sec.~\ref{subsec:nonSync_proof_techniques} shows that the elements of matrix $\genComConnect$ in \eqref{U_poissonBinomial_initDef} essentially denote the probability of having a connection between members of the extended communities (or equivalently between c-node profiles), in the binarized coarse graph. The prior distribution on the extended communities is denoted by the probability vector $\mathbf{s}$. We define the scaled Chernoff-Hellinger (CH) divergence as:
\begin{align}\label{scaled_CH_divergence_def}
    \begin{array}{l}
         \scaledCHdiv(\textrm{diag}(\mathbf{s})\genComConnect_{k},\textrm{diag}(\mathbf{s})\genComConnect_{k'}) 
         \triangleq \displaystyle\max_{0\leq t \leq 1}  \\
         \qquad\displaystyle\sum_{k''\in[\nonSyncDimMatCom{\COvar}]} \mathbf{s}_{k''}[t  \genComConnect_{kk''}
        + (1-t) \genComConnect_{k'k''} - \genComConnect_{kk''}^t \genComConnect_{k'k''}^{(1-t)}], 
    \end{array}
\end{align}
where the original CH divergence is $\CHdiv=\frac{\dimMatC}{\log\dimMatC}\scaledCHdiv$.
The following theorem provides an error bound for community recovery from the coarse graph. 
\begin{theorem}\label{thm:nonsync_UB_error}
Let $\matF\sim \textrm{SSBM}(\dimMatF,\dimMatCom,p,q)$ from which $\matC$ in \eqref{linear_coarsening_model} is measured under the $\protoCovSize$-homogeneous, balanced, and CO-$\COvar$ constraints. $\mathbf{s}$ is a length-$\nonSyncDimMatCom{\COvar}$ probability vector, $\genComConnect_k$ denotes the $k$-th column of matrix $\genComConnect$ defined in \eqref{U_poissonBinomial_initDef}, and $\nonSyncDimMatCom{\COvar}$ is defined in \eqref{size_extended_coms}.
The probability that the $\text{MAP}$ estimator fails to recover the c-node profile matrix $\genSyncMat$ from $\matC$ (up to relabelling of $\genSyncMat$'s columns) is upper-bounded by:
\begin{align}\label{UB_nonSync_MLFailureError}
    \mathbb{P}(\textrm{MAP failure}) \leq  
       \displaystyle\sum_{\substack{k,k'\in[\nonSyncDimMatCom{\COvar}]\\k<k'}}  e^{-\dimMatC \scaledCHdiv(\textrm{diag}(\mathbf{s})\genComConnect_{k},\textrm{diag}(\mathbf{s})\genComConnect_{k'})},
\end{align} 
where $\scaledCHdiv$ is the scaled CH divergence in \eqref{scaled_CH_divergence_def}.
\end{theorem}
The modeling of the coarse graph under the CO-$\COvar$ constraint, i.e. binarization and profile mapping to extended communities sketched before the theorem, makes the binarized coarse graph fit to the general SBM framework in \citep{abbecommunity}. In general SBM, the connection probability between members of the extended communities is no longer symmetric. Rather, this probability differs for each pair of extended communities. This way, the error bound is straightforwardly derived using equations (44) and (47) in \citep{abbecommunity}, while adjusting the notations. 
The rest of the detailed proof techniques for Theorem~\ref{thm:nonsync_UB_error} is elaborated in Sec.~\ref{subsec:nonSync_proof_techniques}.

Theorem~\ref{thm:nonsync_UB_error} demonstrates that, as the connectivity probability among pairs of the extended communities become distant, the recovery error bound improves. 
\begin{remark}\label{remark:CH_dominant_term}
In order to extract interpretable observations from the recovery error bound in Theorem~\ref{thm:nonsync_UB_error}, we examine the dominant term of the CH divergence in \eqref{scaled_CH_divergence_def}.
For each pair of extended communities, $k,k'$, the dominant term corresponds to an extended community $k''$, where the probability of its connectivity to those communities is the most distant.
We derived an estimate for the dominant term in Sec.~\ref{appendix:proof_remark_CH_dominant_term} of the supplementary materials,
which demonstrates the following: the exponent of the error recovery bound (i.e. the CH divergence) increases as $\protoCovSize$ and $|\constintraComProb-\constextraComProb|$ increase (by fixing whichever $\constintraComProb$ or $\constextraComProb$ that is smaller and increasing the other one), or as $\nu$ decreases, while other parameters remain unchanged. 
\end{remark}

In the following we list the observations derived from Theorem~\ref{thm:nonsync_UB_error} and Remark~\ref{remark:CH_dominant_term}:

1. As we increase the measurement size (i.e. $\dimMatC$, {the number of c-nodes}), the error bound decreases. 

2. As the coverage size per measurement (i.e. $\protoCovSize$, {the number of measured fine nodes represented by a c-node}) expands, the failure error bound decreases. 

3. By allowing measurements overlapping with fewer communities {(i.e. increasing the purity of the c-nodes)}, the error bound drops. This intuitively makes sense due to a decrease in complexity.

4. The expansion of the gap between extra- and intra-community probabilities results in a decrease in the error bound. This is intuitively expected since communities become more distinguishable from one another.

Note that the trends listed above are true so long as the prior $\mathbf{s}$ remains unchanged, or does not change such behaviors. We also assume other parameters except for the one mentioned, remain unchanged. Otherwise, we face perturbing multiple parameters simultaneously, which might make the behavior of the error bound unpredictable and heavily depending on the parameter values.

The following corollary characterizes the asymptotic conditions such that the community recovery error, upper-bounded in Theorem~\ref{thm:nonsync_UB_error}, approaches zero.
\begin{corollary}\label{corollary:nonsync_cond}
Let $\matF\sim \textrm{SSBM}(\dimMatF,\dimMatCom,p,q)$ from which $\matC$ in \eqref{linear_coarsening_model} is measured under the $\protoCovSize$-homogeneous, balanced, and CO-$\COvar$ constraints. 
The probability that the $\text{MAP}$ estimator fails to recover the c-node profile matrix $\genSyncMat$ from $\matC$ (up to relabelling of $\genSyncMat$'s columns), for a constant $\Delta>0$, $0<\thresholdMatCin<\frac{1}{\COvar}$, tends to zero as:
\begin{align}\label{UB_nonSync_mainCond}
\begin{array}{c}
\protoCovSize>\frac{\Delta}{\sqrt{\scaleFineComProb}} , \quad
\constintraComProb\neq\constextraComProb, 
\\ 
\Delta^2(\frac{\dimMatC}{\dimMatF})^2<\scaleFineComProb\leq \upConstScaleFineComProb, \quad\dimMatF\geq \dimMatF_0,\quad\dimMatC, \dimMatF\rightarrow\infty
\end{array}.
\end{align}
The constant $\Delta$ is defined in equation \eqref{def_constants2} in Sec.~\ref{appendix:proof_corollary_nonSyncCond} of the supplementary materials. The remaining parameters are assumed to remain fixed. 
\end{corollary}
The condition in \eqref{UB_nonSync_mainCond} is directly derived from the error bound in Theorem~\ref{thm:nonsync_UB_error}, by tending the exponent to infinity resulting in the error to approach zero. The complete proof is sketched in Sec.~\ref{subsec:nonSync_proof_techniques}.

Corollary~\ref{corollary:nonsync_cond} characterizes the impact of coarsening on the community recovery. After the coarse graph is binarized, the connectivity probability between some c-edges reaches very fast to zero, and the rest to one, which facilitates the separation of communities. Moreover,
the measurement coverage size $\protoCovSize$ {(i.e. the number of measured fine nodes combined into a c-node)}, and the graph binarization threshold $\thresholdMatCin$, must satisfy a lower and upper bound, respectively, to allow perfect community recovery of a coarsened graph through its binarization. The recovery conditions derived in Corollary~\ref{corollary:nonsync_cond}, are illustrated in the last column of  Table~\ref{table:compare_recovery_nonsync}, and are compared with those of the classic (non-coarsened) general SBM that exist in the literature. The comparison is made in terms of various scalings of the parameters. The first column exhaustively partitions the scaling of the connection probability of the coarse graph, which can be a function of $\dimMatC,\dimMatF,\protoCovSize$ and denoted by $\scaleCoarseComProb$, for which the second column shows state-of-the-art conditions to allow or disallow exact recovery. In the third column, different scalings of the coarsening coverage size $\protoCovSize$ are considered, where each scaling results in separate recovery conditions demonstrated in the last column.  

\begin{table*}[!htp]
\small
\caption{\small Comparison of the recovery conditions under the CO-$\COvar$ constraint, derived from Corollary~\ref{corollary:nonsync_cond}. All scaling notations (${o}$ for strictly smaller than, and $\strictGreaterBigO$ for strictly greater than, disregarding constants) are defined with respect to $\dimMatC$. 
$\scaleFineComProb$ is the probability scaling of connections in the fine graph,
$Q$ is a constant matrix (i.e. it does not scale with other variables), $\Delta$ is a positive constant,
and $\scaleCoarseComProb$ represents the probability scaling of connections in the coarse graph, which is a function of the fine and coarse graphs sizes, and the coverage size {(i.e. the number of measured fine nodes represented by a c-node)}, to allow for comparison with the classic scenario (i.e. with $\dimMatF=\dimMatC, \protoCovSize=1$).
} \label{table:compare_recovery_nonsync}
\begin{center}
\begin{tabular}{lllll}
$\scaleCoarseComProb$:     &\textbf{Classic (exact) Recovery} & Scaling of  & \textbf{Recovery, This paper} & 
\\ probability scaling   & $\textrm{SBM}(\dimMatC, \mathbf{s}, \genComConnect=Q\scaleCoarseComProbAt{\dimMatC}{1})$  &  coarsening & 
as $\dimMatC,\dimMatF\rightarrow\infty$
& 
\\
 of connections & as $\dimMatC,\dimMatF\rightarrow\infty$ & coverage size, & 
\\ in coarse graph  &   \citep{abbe2015community} & i.e.  $\protoCovSize$  &
&  
\\
\hline \\
${o}(\frac{\log\dimMatC}{\dimMatC})$ &Impossible & ${o}(\frac{1}{\sqrt{\scaleFineComProb}})$ & Impossible & \\
 & & $c_1\frac{1}{\sqrt{\scaleFineComProb}}$ & Possible if $\constintraComProb\neq\constextraComProb, c_1>\Delta, \scaleFineComProb>\Delta^2(\frac{\dimMatC}{\dimMatF})^2 $& \\
& & $\strictGreaterBigO(\frac{1}{\sqrt{\scaleFineComProb}})$ & Possible if $\constintraComProb\neq\constextraComProb, \scaleFineComProb=\strictGreaterBigO((\frac{\dimMatC}{\dimMatF})^2)$& \\
&&&&\\
$\genConstDegScale\frac{\log\dimMatC}{\dimMatC}$    & Possible if $\CHdiv>1$ & ${o}(\sqrt{\frac{\dimMatC}{\log\dimMatC}})$ & Impossible & \\
& & $c_1\sqrt{\frac{\dimMatC}{\genConstDegScale\log\dimMatC}}$ & Possible if $\constintraComProb\neq\constextraComProb, c_1>\Delta,  \scaleFineComProb>\Delta^2(\frac{\dimMatC}{\dimMatF})^2 $& \\
&  & ${\strictGreaterBigO}(\sqrt{\frac{\dimMatC}{\log\dimMatC}})$ & Possible if $\constintraComProb\neq\constextraComProb, \scaleFineComProb=\strictGreaterBigO((\frac{\dimMatC}{\dimMatF})^2)$& \\
&&&&\\
${\strictGreaterBigO}(\frac{\log\dimMatC}{\dimMatC})$  & Possible if $\CHdiv>0$ & ${o}(\frac{1}{\sqrt{\scaleFineComProb}})$ & Impossible & \\
 &  & $c_1\frac{1}{\sqrt{\scaleFineComProb}}$ & Possible if $\constintraComProb\neq\constextraComProb, c_1>\Delta,  \scaleFineComProb>\Delta^2(\frac{\dimMatC}{\dimMatF})^2 $ & \\
  &  & $\strictGreaterBigO(\frac{1}{\sqrt{\scaleFineComProb}})$ & Possible if $\constintraComProb\neq\constextraComProb, \scaleFineComProb=\strictGreaterBigO((\frac{\dimMatC}{\dimMatF})^2)$ & \\
\end{tabular}
\end{center}
\end{table*}


\subsection{Proof Techniques}
\label{subsec:nonSync_proof_techniques}
The community recovery problem under the CO-$\COvar$ constraint refers to the problem of estimating the c-node profile matrix $\genSyncMat$ that corresponds to the weighted adjacency matrix $\matC$ defined in  \eqref{linear_coarsening_model} and measured from $\matF\sim \textrm{SSBM}(\dimMatF,\dimMatCom,p,q)$ under the $\protoCovSize$-homogeneous, balanced, and CO-$\COvar$ constraints. 
This way, $\matC$ is distributed according to \eqref{homogeneous_general_dist_matC} and hence, can be thought of and modeled as a sample of a \emph{weighted} version of the Overlapping general SBM (OSBM) random graph ensemble. The formal definition of general SBM is found in \citep{abbe2015community}. We define the weighted OSBM that models $\matC$, similar to the classic OSBM, except that the node profiles $\genSyncVec_i$ for all $i$ belong to the set $\profileSet$ defined as \eqref{profile_set}.
rather than the set of any length-$\dimMatCom$ binary vectors $\lbrace 0,1\rbrace^{\dimMatCom}$. Furthermore, the weighted OSBM that models $\matC$, an edge between pairs of nodes is distributed as the Poisson Binomial distribution in \eqref{homogeneous_general_dist_matC}, rather than the Bernoulli distribution in classic OSBM.
Note that due to the Community Overlap (i.e. CO-$\COvar$) assumption on $\matC$, the edge distributions depend on the inner product of the pairwise profiles, which takes values between $0$ and $\protoCovSize^2$, i.e. $0\leq \genSyncVec_i^\intercal\genSyncVec_j\leq \protoCovSize^2$. Hence, the weighted OSBM is not symmetric.

Deriving the conditions that allow the community recovery from a weighted OSBM, except for the symmetric case (c.f. Sec.~\ref{subserc:fullSync_mainBody} and Sec.~\ref{appendix:perfect_sync} in the supplementary materials), is an open problem \citep{xu2020optimal}. In the following, we exploit the properties of the special case of the weighted OSBM concerning this study, which enables its transformation to a classic (unweighted) general SBM. 
We propose a two-stage strategy that first binarizes $\matC$ and then represents the resultant unweighted OSBM as an unweighted classic (non-overlapping) general SBM. The binarization is motivated for two reasons. First, binarization is widely used to simplify and sparsify weighted graphs. Second, through binarization, we can leverage existing work in community detection literature to study the conditions to recover the c-node profile matrix. 

\subsubsection{Stage one: Binarizing $\matC$}\label{subsec:stageOne_binarization}
The summation in the coarsening model \eqref{linear_coarsening_model} suggests the concentration of edge weights around a mean value. Hence, for the c-edges that corresponds to a pair of c-nodes measuring from only one community, the expectation of the weights tend to concentrate about means $\intraComProb\protoCovSize^2$ or $\extraComProb\protoCovSize^2$. Regarding the c-nodes measuring from multiple communities, their corresponding c-edge weights concentrate about means $\intraComProb \genSyncVec_i^\intercal\genSyncVec_j+ \extraComProb(\protoCovSize^2-\genSyncVec_i^\intercal\genSyncVec_j)$. This motivates solving our weighted OSBM problem by first binarizing $\matC$. Such binarization facilitates community recovery by adopting the much more evolved tools available for unweighted graphs. 
We define the binarized coarse measured matrix $\biMatC$ as:
\begin{align}\label{matC_binarized}
    \biMatC_{ij} \triangleq \left\lbrace\begin{array}{ll}
        1 &  \textrm{ if } \matC_{ij} \geq \protoCovSize^2(\thresholdMatCin \intraComProb + (1-\thresholdMatCin)\extraComProb)\\
        0 & \textrm{ else}
    \end{array}\right.,
\end{align}
for $0 \leq \thresholdMatCin \leq 1$. 
The chosen threshold, i.e. $\protoCovSize^2(\thresholdMatCin \intraComProb + (1-\thresholdMatCin)\extraComProb)$ in \eqref{matC_binarized}, is a suitable choice since it is lower- and upper- bounded by $\extraComProb\protoCovSize^2, \intraComProb\protoCovSize^2$, the minimum and maximum mean values of $\matC_{ij}$ for various profile inner products. This way, we only keep the most significant edges, i.e. those whose weights are above the mean value of the intra-community connections. 
\subsubsection{Stage two: SBM representation of the OSBM}
Through the binarization explained in Sec.~\ref{subsec:stageOne_binarization}, the coarse graph $\matC$ previously modeled as a weighted general OSBM, is converted to $\biMatC$, which is a classic (unweighted) general OSBM.
Following the approach suggested in \citep{abbe2017community}, we convert the classic OSBM to an equivalent \textit{non-overlapping} general SBM. To do so, instead of the original community set $[\dimMatCom]$, we use the extended community set $[\nonSyncDimMatCom{\COvar}]$ , where $\nonSyncDimMatCom{\COvar} \triangleq |\profileSet|$ defined in \eqref{size_extended_coms}, where each extended community represents a possible c-node profile $\genSyncVec_i\in\profileSet$ for all $i\in[\dimMatC]$. 
The one-to-one function $h:\profileSet \xrightarrow[]{}[\nonSyncDimMatCom{\COvar}]$ provides indexing for the extended communities, i.e. a profile vector $\genSyncVec_i\in\profileSet$ maps to an extended community $k = h(\genSyncVec_i)$.
Such conversion of profiles to extended communities, models the binarized matrix of measurements $\biMatC$ in \eqref{matC_binarized} as a general unweighted SBM denoted by $\biMatC \sim \textrm{SBM}(\dimMatC, \mathbf{s}, \genComConnect)$, where $\mathbf{s}$ is a prior probability vector of the extended communities. 
Sec.~\ref{appendix:more_nonSync_proof_techniques} in the supplementary materials provides the formal definition of the general unweighted SBM, the derivation of the matrix of community connectivity probabilities $\genComConnect$, and the remaining of the proof techniques of Theorem~\ref{thm:nonsync_UB_error} and Corollary~\ref{corollary:nonsync_cond}.


\subsection{Stronger recovery under the special Community Overlap (CO)-$1$ constraint}\label{subserc:fullSync_mainBody}
The results in Sec.~\ref{subsec:mainResultsgeneral_sync} are applicable to coarse measured graphs under the general CO-$\nu$ constraint, for all $1\leq \nu\leq \dimMatCom$. However, the CO-$1$ constraint is an special case, which corresponds to a weighted and \textit{symmetric} SBM. 
Contrary to the general (i.e. non-symmetric) weighted SBM model which is an open problem, the community recovery from such weighted and symmetric SBM has already been addressed in the literature \citep{jog2015information,xu2020optimal}. In the following theorem, we adopt the results of \citep{jog2015information} to achieve stronger recovery conditions under the CO-$1$ constraint, compared with those of the general CO-$\COvar$ scenario in Corollary~\ref{corollary:nonsync_cond}.
\begin{theorem}\label{thm:fullsync_UB_error_cond}
Let $\matF\sim \textrm{SSBM}(\dimMatF,\dimMatCom,\intraComProb,\extraComProb)$ from which $\matC$ in \eqref{linear_coarsening_model} is measured under the $\protoCovSize$-homogeneous, and CO-$1$ constraints. The probability that the $\text{MAP}$ estimator fails to recover the 
c-node profile matrix $\genSyncMat$ from $\matC$ (up to relabelling of $\genSyncMat$'s columns)
from $\matC$, tends to zero as:
\begin{align}\label{condition_exact_rec_regime} 
    \left\lbrace\begin{array}{ll}
    \constintraComProb\neq\constextraComProb &\mathrm{if} \quad \dimMatC<\infty \quad \protoCovSize\sqrt{\scaleFineComProb}\rightarrow\infty 
    \\
  \frac{\constintraComProb+\constextraComProb}{2}-\sqrt{\constintraComProb\constextraComProb}>\displaystyle\lim_{\dimMatC\rightarrow\infty}&\left[\frac{\dimMatCom}{2}\frac{\log\dimMatC}{\protoCovSize^2\scaleFineComProb\dimMatC}\right] \quad\mathrm{if} \quad \dimMatC,\dimMatF\rightarrow\infty,
    \end{array}\right.
\end{align}
if $\scaleFineComProb\xrightarrow[]{\dimMatF\rightarrow\infty} 0$. $\dimMatCom$ is assumed to remain fixed.
\end{theorem}
More explanations and proof details are found in Sec.~\ref{appendix:perfect_sync} and Sec.~\ref{appendix:proof_theorm_fullSyncUBCond} of the supplementary materials.


\section{Numerical Results}\label{sec:numerical_results}
In this section, we evaluate the error behavior of the community recovery from synthetically generated coarse measured graphs. We compare the theoretical error bounds derived in Sec.~\ref{sec:general_sync}, with state-of-the-art community detection methods from existing works that are applied to the generated coarse graphs. {\footnote{ 
The Python code to reproduce the results of this paper is available at:  \url{https://github.com/NaGho/Community-Detection-From-Coarse-Measured-Graphs}.
}}
{
It should be noted these algorithmic methods only output the index of the nodes estimated to be assigned. This translates into the recovery of a binarized version of the community assignment matrix $\syncComAs$.
}
Refer to Sec.~\ref{appendix:simulations_methodology} in supplementary materials for the detailed methodology used in this section. 


In Fig.~\ref{fig.numerical_CommunityRecoveryEval}, the theoretical error bound (solid line), as well as the community recovery error for multiple state-of-the-art overlapping community detection methods \citep{rossetti2019cdlib} are plotted\footnote{{
The results in this section are computed assuming  $\intraComProb,\extraComProb,\dimMatCom,\protoCovSize$ are known. However, using model selection methods, heuristics can be developed to estimate these parameters when they are not known apriori.
}}. The methods include Modularized non-negative matrix factorization (M-NMF) \citep{wang2017community}, Speaker-listener Label Propagation Algorithm (SLPA) \citep{xie2011slpa,xie2013overlapping}, Non-Negative Symmetric Encoder-Decoder (NNSED) \citep{sun2017non}, and Cluster Affiliation Model for Big Networks (BigClam) \citep{yang2013overlapping}
(dashed lines). Note that we have evaluated these methods for various hyper-parameters and plotted their best performance.

\begin{figure}[h]
\centering
\subfloat[w.r.t. $\dimMatC$ for $\protoCovSize=50$.]{\includegraphics[width=21em]{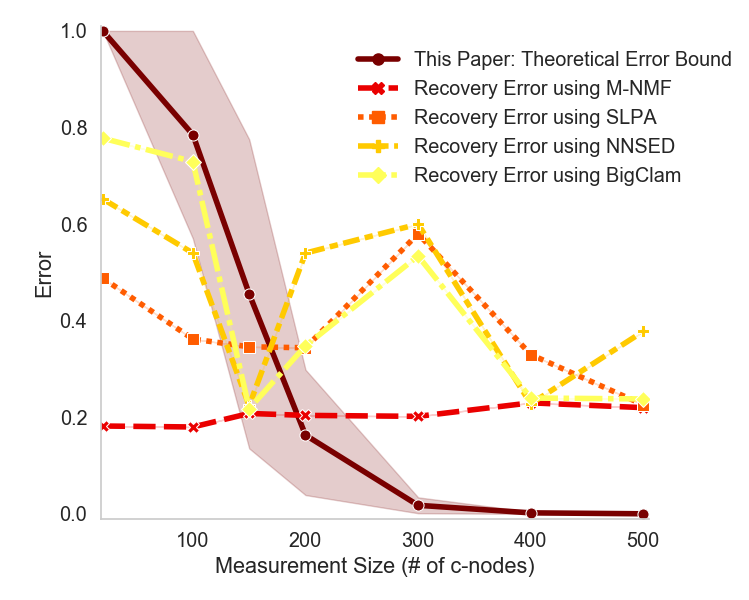}
\label{fig.CommunityRecoveryEval_wrt_measurementSize}}
\hfil
\subfloat[w.r.t. $\protoCovSize$  for $\dimMatC=400$.]{\includegraphics[width=20em]{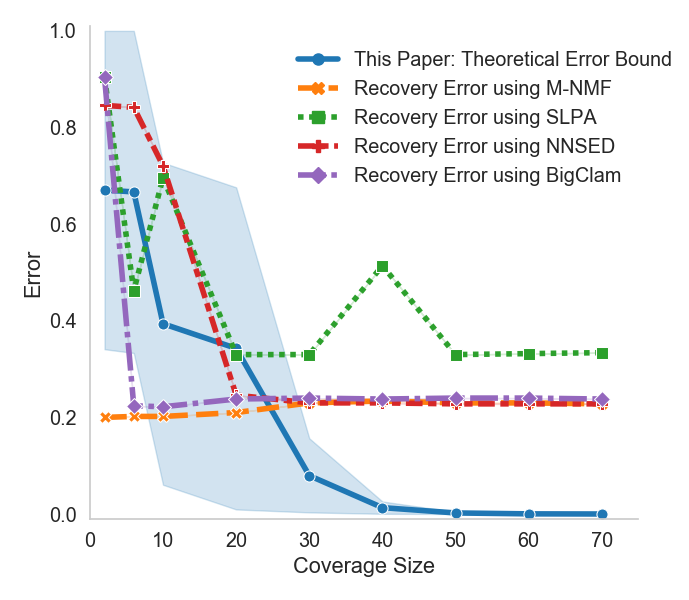}%
\label{fig.CommunityRecoveryEval_wrt_coverageSize}}
\caption{\small Community recovery error for $\dimMatF=30000$ fine nodes, $\COvar=2$ community overlap (CO), $\dimMatCom=5$ communities, and $\constintraComProb=500,\constextraComProb=50$ intra- and extra- community constants for the probability of connectivity.}
\label{fig.numerical_CommunityRecoveryEval}
\end{figure}

From Fig.~\ref{fig.CommunityRecoveryEval_wrt_measurementSize}, we observe that as we increase the measurement size (i.e. $\dimMatC$, {the number of c-nodes}), the theoretical error bound drops monotonically. 
Similarly, Fig.~\ref{fig.CommunityRecoveryEval_wrt_coverageSize} plots the community recovery error with respect to the coverage size $\protoCovSize$ {(i.e. the number of measured fine nodes combined into a c-node)}, demonstrating that increasing the coverage size monotonically improves the theoretical community recovery error. 
These observations confirm the expectations made subsequent to Theorem~\ref{thm:nonsync_UB_error}.
Although the simulated methods, both in Fig.~\ref{fig.CommunityRecoveryEval_wrt_measurementSize} and Fig.~\ref{fig.CommunityRecoveryEval_wrt_coverageSize}, do not perform as predictable as the theoretical error bound, most of them show an overall decrease in their recovery error when respectively, the number of measurements and the coverage size increase.
Note that the light shade in Fig.~\ref{fig.numerical_CommunityRecoveryEval} around the theoretical bound represents the ambiguity in the calculation of the bound (c.f. Sec.~\ref{appendix:simulations_methodology} of the supplementary materials). 

Note that the theoretical bound is the \textit{upper bound} for the $\text{MAP}$ estimator. Fig.~\ref{fig.numerical_CommunityRecoveryEval} shows that the upper bound seem to be loose in certain regimes (e.g. for small $\dimMatC,\protoCovSize$), in which existing methods perform better. However, as the measurement- and the coverage sizes increase, the theoretical error bound becomes tight and outperforms existing community detection methods with an increasing gap.

\section{Conclusion and Future Work}
We introduced a mathematical framework based on the stochastic block model, that characterizes community recovery from coarse measured graphs. We developed theoretical conditions, on the quantity and properties of the measurements with respect to the community structure of the high-resolution graph, to achieve perfect recovery. The assumptions of homogeneous and balanced measurements were essential to this work. We leave to future work the relaxation of these assumptions.
Moreover, community recovery in a coarse measured graph, in which communities modeled using the weighted and overlapping stochastic block model, utilized edge weight binarization. Future work can look into community recovery without binarization, in which one would use full graph weight distribution for recovery.

Finally, a significant gap was observed between the performance of state-of-the-art community detection algorithms, with the theoretical error bounds derived in this paper, in certain regimes. 
This gap motivates future work to improve existing clustering algorithms to achieve its theoretical potential. 
{
An algorithmic investigation into recovery performance, e.g. similar to the variational inference approaches used in \citep{aicher2015learning,dulac2020mixed}, is a promising direction to future work and would complement our theoretical analyses. 
}
{
\section*{Acknowledgements}
This work was supported by the Natural Sciences and Engineering Research Council (NSERC) of Canada through a Discovery Research Grant; the National Science Foundation Grants CCF-2029044 and CCF-2048223; the National Institutes of Health Grant 1R01GM140468-01; Connaught International Scholarship for Doctoral Students; and the Vector Postgraduate Affiliate Award by Vector Institute for AI.
}
\bibliography{refs}

\setcounter{section}{0}

\onecolumn
\aistatstitle{Supplementary Materials}
\section{More on Stronger recovery under the special Community Overlap (CO)-$1$ constraint in Sec.~\ref{subserc:fullSync_mainBody}}\label{appendix:perfect_sync}

In the CO-$1$ constrained measurements with respect to a graph $G\in\mathcal{G}(\mathcal{V},E)$, each c-node only measures from one community. Hence, the profile \textit{matrix} $\genSyncMat$ defined in \eqref{gen_profile_def} has row-wise support of size $1$ and can be re-expressed as a profile \textit{vector} $\fullComAs\in [\dimMatCom]^{\dimMatC}$, where for all $i\in[\dimMatC]$, $\fullComAs_i$ identifies the index of the community with which the $i$th c-node aligns:
\begin{align}\label{full_comAs_def}
    \fullComAs_i= k \quad \textrm{if}\quad \genSyncMat_{ik}\neq 0
\end{align}
This way, the statics of $\matC$ is simplified as the following Lemma.
\begin{lemma}
Let $\matF\sim \textrm{SSBM}(\dimMatF,\dimMatCom,\intraComProb,\extraComProb)$ from which $\matC$ in \eqref{linear_coarsening_model} is measured under the $\protoCovSize$-homogeneous and CO-$1$ measurement constraints. Then, $\matC_{ij}$'s are i.i.d. random variables for all $i>j$, with distribution:
\begin{align}\label{matcDist_perfSync}
        \matC_{ij} \sim \left\lbrace \begin{array}{ll}
        \textrm{Binomial}(\protoCovSize^2,\intraComProb)& \textrm{if } \fullComAs_i=\fullComAs_j \\
        \textrm{Binomial}(\protoCovSize^2,\extraComProb) & \textrm{else}
    \end{array}\right..
\end{align}
\end{lemma}
\begin{proof}
The result is straightforwardly derived from Lemma \ref{lemma:coarse_mat_stat} and the connection between the c-node profile matrix and vector in \eqref{full_comAs_def}.
\end{proof}

Assuming a uniform prior over the set $[\dimMatCom]$ on the $\fullComAs$'s elements, the community recovery problem is reduced to estimating the c-node profile vector $\fullComAs$ (i.e. assigning an element of $[\dimMatCom]$ to each c-node) from a graph with weighted adjacency matrix $\matC$ of \eqref{linear_coarsening_model} that is measured from $\matF\sim \textrm{SSBM}(\dimMatF,\dimMatCom,\intraComProb,\extraComProb)$ under the $\protoCovSize$-homogeneous, balanced, and CO-$1$ constraints.

Let an $\text{MAP}$ estimator take a measured graph $\tilde{G}$ with the true c-node profile vector $\fullComAs$, and estimate $\hat{\fullComAs}$ which assigns an element of $[\dimMatCom]$ to every c-node in $\tilde{G}$. In the following, we study the conditions such that the probability of failure, i.e. assigning a wrong community to at least one c-node considering the relabelling of communities, approaches $0$ when the intra- and extra-community distribution of associations are $\textrm{Binomial}(\protoCovSize^2,\intraComProb)$ and $\textrm{Binomial}(\protoCovSize^2,\extraComProb)$, respectively. 

As explained before, 
the $\matC$ distributed as \eqref{matcDist_perfSync} can be thought of as a graph representation modeled as a weighted variant of the SSBM (WSSBM). The problem of community recovery from the WSSBM is addressed in \citep{jog2015information,xu2020optimal}.
The following Lemma uses { Theorem 3.2} in \citep{jog2015information} to find an upper bound on the probability of $\text{MAP}$ failure in estimating the c-node profiles.

\begin{lemma}\label{lemma:fullsync_UB_error}
Let $\matF\sim \textrm{SSBM}(\dimMatF,\dimMatCom,\intraComProb,\extraComProb)$ from which $\matC$ in \eqref{linear_coarsening_model} is measured under the $\protoCovSize$-homogeneous and CO-$1$ constraints. The probability that the $\text{MAP}$ estimator 
fails to recover the c-node profile vector $\fullComAs\in [K]^{\dimMatC}$ from $\matC$ (up to relabelling of the indices), is upper-bounded by:
\begin{align}\label{UB_MLFailureError}
    \begin{array}{c}
       \mathbb{P}(\textrm{MAP failure}) \leq \displaystyle\sum_{m=1}^{\lfloor\frac{\dimMatC}{2\dimMatCom}\rfloor} \min\lbrace (\frac{e \dimMatC\dimMatCom}{m})^m, \dimMatCom^{\dimMatC}\rbrace e^{(-\frac{\dimMatC}{\dimMatCom} m+m^2)I_{}} \\
         \qquad\qquad\qquad + \displaystyle\sum_{m=\lfloor\frac{\dimMatC}{2\dimMatCom}\rfloor+1}^{\dimMatC} \min\lbrace (\frac{e \dimMatC\dimMatCom}{m})^m, \dimMatCom^{\dimMatC}\rbrace e^{-\frac{2m \dimMatC}{9\dimMatCom}I_{}}
    \end{array},
\end{align}
where 
\begin{align}\label{renyi_div_binomial}
    I_{} \triangleq -2 \protoCovSize^2 \log\left(\sqrt{(1-\intraComProb)(1-\extraComProb)}+\sqrt{\intraComProb\extraComProb} \right).
\end{align}
\end{lemma}
The error bound in \eqref{UB_MLFailureError} is directly taken from { Theorem 3.2} of \citep{jog2015information} after some notation adjustments. Also, we derived the Renyi divergence in \eqref{renyi_div_binomial} in a closed from using the distributions in \eqref{matcDist_perfSync}. A complete proof can be found in Sec.~\ref{appendix:proof_lemma_fullSyncUBError} of the supplementary materials.

Lemma \ref{lemma:fullsync_UB_error} not only helps us characterize and predict the behaviour of the community recovery error, but also allows studying the conditions on the parameters of the underlying generative SSBM on $\matF$, that asymptotically guarantee such community recovery. This has been investigated in Theorem~\ref{thm:fullsync_UB_error_cond}, in which the recovery of the c-node profile matrix $\genSyncMat$ in \eqref{gen_profile_def} is equivalent to the recovery of the c-node profile vector $\fullComAs\in [\dimMatCom]^{\dimMatC}$ defined in \eqref{full_comAs_def}.

For better visualization and understanding, the recovery conditions in Theorem~\ref{thm:fullsync_UB_error_cond} is illustrated in the last column of Table~\ref{table:compare_recovery_fullsync}. The table compares the recovery conditions of a coarse measured graph with those of the classic SBM existing in the literature, in terms of various scaling of the parameters. The first column provides exhaustive partitioning of the connection probability scaling of the fine graph $\scaleFineComProb$, for which the second column shows state-of-the-art conditions to allow exact recovery. In the third column, and corresponding to each fine connection probability scaling, different scalings of the coarsening coverage size $\protoCovSize$ {(i.e. the number of measured fine nodes combined into a c-node)} are considered that will result in separate recovery conditions shown in the last column. 
The additional constraint on the $\scaleFineComProb$ scaling in the last column comes from the inequality $\protoCovSize\leq\frac{\dimMatF}{\dimMatC}$ from the definition of the coverage size $\protoCovSize$ in Def.~\ref{def_homogenous}.

From Lemma \ref{lemma:fullsync_UB_error} and its consequences demonstrated in Table~\ref{table:compare_recovery_fullsync}, a significant observation can be derived. When the connection probability scaling of an observed graph is less than logarithmic ${o}(\frac{\log\dimMatC}{\dimMatC})$, a classic (uncoarsened) community recovery is impossible \citep{abbe2015community}. However, graph coarsening allows for community recovery, by compensating via large coverage size. This is significant since in classic graphs with constant-size measurements, tending community recovery error to zero is impossible, while it is possible for coarsened graphs with coverage size growing as $\strictGreaterBigO(\sqrt{\frac{\log\dimMatC}{\scaleFineComProb\dimMatC}})$. Furthermore, for fixed-size coverage while tending the measurement-size to infinity, the $\constintraComProb,\constextraComProb$ gap to allow recovery becomes less harsh, as the coverage size becomes larger, compared with the classic version ($\protoCovSize=1$). 


\begin{table}[!h]
\small
\caption{\small Comparison of the recovery conditions under the CO-$1$ constraint, derived from Theorem~\ref{thm:fullsync_UB_error_cond}. All scaling notations (${o}$ for strictly smaller, $\Theta$ for equal, and $\strictGreaterBigO$ for strictly greater orders) are defined with respect to $\dimMatC$. 
$\scaleFineComProb$ is the probability scaling of connections in the fine graph, 
$\scaleCoarseComProb$ represents the probability scaling of connections in the coarse graph, which is a function of the fine and coarse graphs sizes, and the coverage size to allow for comparison with the classic scenario (i.e. with $\dimMatF=\dimMatC, \protoCovSize=1$).} \label{table:compare_recovery_fullsync}
\begin{center}
\begin{tabular}{lllll}
$\scaleCoarseComProb$:  &\textbf{Classic (exact) Recovery} &  Scaling of  & \textbf{Recovery, This paper} & 
\\  probability scaling  & SSBM$(\dimMatC,\constintraComProb\scaleCoarseComProbAt{\dimMatC}{1},\constextraComProb\scaleCoarseComProbAt{\dimMatC}{1})$  &  coarsening  & 
\\ of connections & \citep{zhang2016minimax} &  coverage size $\protoCovSize$ &
&
\\
in coarse graph &&&&\\
\hline \\
${o}(\frac{\log\dimMatC}{\dimMatC})$ &Impossible & ${o}(\sqrt{\frac{\log\dimMatC}{\scaleFineComProb\dimMatC}})$  & Impossible&  \\
& & $\genConstDegScale\sqrt{\frac{\log\dimMatC}{\scaleFineComProb\dimMatC}}$ &  Possible if $\frac{\constintraComProb+\constextraComProb}{2}-\sqrt{\constintraComProb\constextraComProb}\geq\frac{\dimMatCom}{2\genConstDegScale^2},$& \\
&&& $\qquad\qquad\textrm{and }\scaleFineComProb\geq \genConstDegScale^2\frac{\dimMatC\log\dimMatC}{\dimMatF^2}, \dimMatC,\dimMatF\rightarrow\infty$\\
& & $\strictGreaterBigO(\sqrt{\frac{\log\dimMatC}{\scaleFineComProb\dimMatC}})$ &Possible if $\constintraComProb\neq\constextraComProb, \scaleFineComProb=\strictGreaterBigO(\frac{\dimMatC\log\dimMatC}{\dimMatF^2})$ & \\
&&&&\\
&&&&\\
$\genConstDegScale\frac{\log\dimMatC}{\dimMatC}$    &Possible if $\frac{\constintraComProb+\constextraComProb}{2}-\sqrt{\constintraComProb\constextraComProb}\geq\frac{\dimMatCom}{2\genConstDegScale},$ &${o}(1)$  & Impossible & \\
& $\dimMatC,\dimMatF\rightarrow\infty$ & $\Theta(1)$ &  Possible if $\frac{\constintraComProb+\constextraComProb}{2}-\sqrt{\constintraComProb\constextraComProb}\geq\frac{\dimMatCom}{2\genConstDegScale\protoCovSize^2},$& \\
&&& $\qquad\qquad\textrm{and }\scaleFineComProb\geq \protoCovSize^2\frac{\dimMatC\log\dimMatC}{\dimMatF^2}, \dimMatC,\dimMatF\rightarrow\infty$\\
&  & $\strictGreaterBigO(1)$ & Possible if $\constintraComProb\neq\constextraComProb, \scaleFineComProb=\strictGreaterBigO(\frac{\dimMatC\log\dimMatC}{\dimMatF^2})$ & \\
&&&&\\
&&&&\\
${\strictGreaterBigO}(\frac{\log\dimMatC}{\dimMatC})$   & Possible if $\constintraComProb\neq\constextraComProb,\dimMatC,\dimMatF\rightarrow\infty$ & ${o}(\sqrt{\frac{\log\dimMatC}{\scaleFineComProb\dimMatC}})$ & Impossible & \\
& & $\genConstDegScale\sqrt{\frac{\log\dimMatC}{\scaleFineComProb\dimMatC}}$ &  Possible if $\frac{\constintraComProb+\constextraComProb}{2}-\sqrt{\constintraComProb\constextraComProb}\geq\frac{\dimMatCom}{2\genConstDegScale^2},$ & \\
&&& $\qquad\qquad\textrm{and }\scaleFineComProb\geq \genConstDegScale^2\frac{\dimMatC\log\dimMatC}{\dimMatF^2}, \dimMatC,\dimMatF\rightarrow\infty$\\
& & $\strictGreaterBigO(\sqrt{\frac{\log\dimMatC}{\scaleFineComProb\dimMatC}})$ & Possible if $\constintraComProb\neq\constextraComProb, \scaleFineComProb=\strictGreaterBigO(\frac{\dimMatC\log\dimMatC}{\dimMatF^2})$ & \\
\end{tabular}
\end{center}
\end{table}

\section{More on Proof Techniques in Sec.~\ref{subsec:nonSync_proof_techniques}}\label{appendix:more_nonSync_proof_techniques}
We started the proof of Theorem~\ref{thm:nonsync_UB_error} and Corollary~\ref{corollary:nonsync_cond} by modeling the $\genSyncMat$ recovery problem under the CO-$\COvar$ constraints, as a community detection problem from a weighted OSBM. Next, we used a mapping to convert the problem to a community detection from a general unweighted SBM model. This conversation facilitates finding the estimation error bounds  in Theorem~\ref{thm:nonsync_UB_error} and the recovery conditions in Corollary~\ref{corollary:nonsync_cond}.

Before deriving the matrix of community connectivity probabilities $\genComConnect$ in Lemma~\ref{lemma:SBM_genSyncMat}, we first rewrite the formal definition for the general unweighted SBM, adopted from { Definition 1 in} \citep{abbe2017community}, in the following remark.
\begin{remark}
Let $V$ be a positive integer (the number of vertices), $\kappa$ be a positive integer (the number of communities), $\mathbf{s} = (\mathbf{s}_1,\cdots, \mathbf{s}_\kappa)$ be a probability vector on $[d]$ (the prior on the $d$ communities) and $A$ be a $\kappa$ by $d$ symmetric matrix with entries in $[0, 1]$
(the connectivity probabilities). The pair $(X,G)$ is drawn under $\textrm{SBM}(V,\mathbf{s},A)$ if $X$ is a $V$-dimensional
random vector with i.i.d. components distributed under $\mathbf{s}$, and $G$ is a $V$-node graph where vertices $i$ and $j$ are connected with probability $A_{X_i},A_{X_j}$, independently of other pairs of vertices.
\end{remark}
For notation simplicity in the following, rather than considering the pair $(X,G)$ being drawn from $\textrm{SBM}(V,\mathbf{s},A)$, we only use the adjacency matrix which is a representative notation for $G$.

\begin{lemma}\label{lemma:SBM_genSyncMat}
Let $\matF\sim \textrm{SSBM}(\dimMatF,\dimMatCom,p,q)$ from which $\matC$ in \eqref{linear_coarsening_model} is measured under the $\protoCovSize$-homogeneous, balanced, and CO-$\COvar$ constraints. Let $\matC$ binarize to $\biMatC$ matrix as in \eqref{matC_binarized}. 
$\biMatC$ is distributed as $\biMatC \sim \textrm{SBM}(\dimMatC, \mathbf{s}, \genComConnect)$, with $\mathbf{s}$ be the length-$\nonSyncDimMatCom{\COvar}$ vector prior distribution for the c-nodes' extended community profile vector, and for all $i,j\in[\dimMatC]$ and $\genSyncVec_i,\genSyncVec_j\in \profileSet$:
\begin{align}\label{U_poissonBinomial}
    \begin{array}{l}
       \genComConnect_{h(\genSyncVec_i),h(\genSyncVec_j)} =
        \mathbb{P}(X\geq \protoCovSize^2(\thresholdMatCin p + (1-\thresholdMatCin)q))
    \end{array},
\end{align}
for $X \sim \textrm{PoissonBinomial}(\lbrace p\rbrace^{\genSyncVec_i^\intercal\genSyncVec_j}, \lbrace q \rbrace^{\protoCovSize^2-\genSyncVec_i^\intercal\genSyncVec_j})$.
\end{lemma}
The proof is found in Sec.~\ref{appendix:binarized_SBM_dist} in supplementary materials.
\begin{remark}\label{normalized_profile_ind}
The $\protoCovSize$-normalized c-node profile vectors, i.e. $\frac{\genSyncVec_i}{\protoCovSize}$ for all $i\in[\dimMatC]$, corresponding to $\protoCovSize$-homogeneous, balanced, and CO-$\COvar$ constrained measurements, take values in the set $\lbrace \frac{1}{\COvar},\frac{1}{\COvar-1},\cdots,\frac{1}{2},1\rbrace$ and are independent of $\dimMatF,\dimMatC,\protoCovSize$. Moreover, $\protoCovSize$ is assumed to be divisible by $\lbrace 2, \cdots,\COvar \rbrace$.
\end{remark}
The remaining proof of Corollary~\ref{corollary:nonsync_cond} is elaborated in Sec.~\ref{appendix:proof_corollary_nonSyncCond} in supplementary materials.

\section{Methodology of the Numerical Results in Sec.~\ref{sec:numerical_results}}\label{appendix:simulations_methodology}
The numerical results is essentially comprised of two parts: 
\begin{itemize}
    \item \textbf{Calculating the theoretical error bound derived in Theorem~\ref{thm:nonsync_UB_error}}:
    The bound computes \eqref{UB_nonSync_MLFailureError}, for which we need to calculate $\genComConnect$. From Lemma~\ref{lemma:SBM_genSyncMat}, each element of $\genComConnect$ is a cumulative Poisson Binomial distribution which is intractable to compute for our parameter values. Hence, we calculate the mean, as well as the lower and upper bounds of $\genComConnect$ elements derived in Lemma~\ref{lemma:SBM_genSyncMat_ULBounds}, to take into the ambiguity of the $\genComConnect$ evaluation.
    \item \textbf{Evaluating the performance of existing state-of-the art community detection methods on synthetically generated graphs}: we generate fine graphs $\matF$, with their corresponding community assignment matrices $\comFine$, using the SBM random graph generators in \textit{networkX} Python module.
    We fix $\COvar$ and $\mathbf{s}$, the prior on the extended communities, based on which a profile matrix $\genSyncMat$ was randomly generated under the $\protoCovSize$-homogeneous, balanced, and CO-$\COvar$ constraints. Next, we coarse measure the fine graph according to \eqref{linear_coarsening_model}, using a randomly generated $\matLinSens$ (using profile matrix $\genSyncMat$ and the community assignment matrix $\comFine$ of the initially generated graph). 
    Note that as explained in Sec.~\ref{sec:general_sync}, we utilized the general SBM framework to \textit{characterize} and derive error bounds for the recovery of communities from the binarized, SBM-represented coarse graphs. The \textit{algorithm} for such community recovery does not yet have an efficient implementation \citep{abbe2015community}. We use the following four existing state-of-the-art \textit{overlapping} community detection methods that are applied to the generated, and later binarized, coarse graphs: 
    \begin{enumerate}
        \item Modularized Non-Negative Matrix Factorization (M-NMF) \citep{wang2017community},
        \item Speaker-listener Label Propagation Algorithm (SLPA) \citep{xie2011slpa},
        \item Non-Negative Symmetric Encoder-Decoder (NNSED) \citep{sun2017non},
        \item Cluster Affiliation Model for Big Networks (BigClam) \citep{yang2013overlapping}.
    \end{enumerate}
    We used CDLIB python module with the implementations of these algorithms \citep{rossetti2019cdlib}. 
    To evaluate the goodness of recovery, we use the $1-$``nF1'' measure, i.e. normalized F1 subtracted from $1$ \citep{rossetti2016novel}, to evaluate the overlapping community detection error. nF1 is considered a standard and computationally tractable community evaluation measure, also implemented as part of the CDLIB module.
\end{itemize}

\section{Remaining Proofs}
\subsection{Proof of Lemma \ref{lemma:coarse_mat_stat}}\label{appendix:proof_lemma_coarse_mat_stat}
\begin{proof}
Given the SSBM generative model for $\matF$ and the measurement matrix $\matLinSens$, each element in $\matC$ in \eqref{linear_coarsening_model} is the sum of independent Bernoulli random variables with $p$ and $q$ success probabilities. Since $\mathbf{b}_i$'s are disjointed, $\matC_{ij}$'s become independent random variables for all $i>j$ with distribution:
\begin{align}\label{general_dist_matC}
    \begin{array}{ll}
        \matC_{ij} \sim
         & \textrm{PoissonBinomial}(\lbrace  \intraComProb\rbrace^{\genSyncVec_i^\intercal\genSyncVec_j}, \lbrace \extraComProb \rbrace^{|\textrm{supp}(\mathbf{b}_i)| |\textrm{supp}(\mathbf{b}_j)|-\genSyncVec_i^\intercal\genSyncVec_j})
    \end{array}.
\end{align}
Note that \eqref{general_dist_matC} generally holds for the measurement matrices defined prior to \eqref{linear_coarsening_model}.
Under the $\protoCovSize$-homogeneous measurement assumption defined in Def.~\ref{def_homogenous}, for all $i\in[\dimMatC]$ we have:
\begin{align}
    |\textrm{supp}(\mathbf{b}_i)|=\protoCovSize.
\end{align}
Hence, equation \eqref{general_dist_matC} simplifies to \eqref{homogeneous_general_dist_matC}.
Both \eqref{general_dist_matC} and \eqref{homogeneous_general_dist_matC} show the likelihoods of $\matC$ elements given $\matLinSens$ and $\genSyncMat$. 
\end{proof}


\subsection{Proof of Remark~\ref{remark:CH_dominant_term}}\label{appendix:proof_remark_CH_dominant_term}
\begin{proof}
We intend to find a lower bound on $\scaledCHdiv$ in \eqref{scaled_CH_divergence_def} whose dominant term is sufficiently simple to derive interpretable observations from Theorem~\ref{thm:nonsync_UB_error}. 
We start with the definition in \eqref{scaled_CH_divergence_def}, by fixing the parameter $t=\frac{1}{2}$:
\begin{align}\label{CH_LB1}
    \begin{array}{ll}
         \scaledCHdiv(\textrm{diag}(\mathbf{s})\genComConnect_{k},\textrm{diag}(\mathbf{s})\genComConnect_{k'}) & \\
         \qquad \geq \displaystyle\sum_{k''\in[\nonSyncDimMatCom{\COvar}]} 
         \mathbf{s}_{k''}[
        \frac{\genComConnect_{kk''}+ \genComConnect_{k'k''}}{2} - \sqrt{\genComConnect_{kk''}\genComConnect_{k'k''}} 
         ]&
         \\
         \qquad\geq \displaystyle\max_{k''\in[\nonSyncDimMatCom{\COvar}]} 
         \mathbf{s}_{k''}[
        \frac{\genComConnect_{kk''}+ \genComConnect_{k'k''}}{2} - \sqrt{\genComConnect_{kk''}\genComConnect_{k'k''}} 
         ]
         \\
         \qquad = \displaystyle\max_{k''\in[\nonSyncDimMatCom{\COvar}]} 
         \mathbf{s}_{k''}
        \frac{(\sqrt{\genComConnect_{kk''}}- \sqrt{\genComConnect_{k'\hat{k}}})^2}{2}  &     
    \end{array}.
\end{align}

To continue, we provide lower and upper bounds for the $\genComConnect$ elements in the following Lemma.
\begin{lemma}\label{lemma:SBM_genSyncMat_ULBounds}
The elements of the extended community connectivity matrix $\genComConnect$ defined in \eqref{U_poissonBinomial_initDef} can be upper and lower bounded, for all $\mathbf{a}, \mathbf{a}'\in \profileSet$, by:
\begin{align}\label{small_deviation_bound}
    \begin{array}{ll}
        |\genComConnect_{h(\mathbf{a}),h(\mathbf{a}')} - \Psi(\frac{\protoCovSize(p-q)((\frac{\mathbf{a}}{\protoCovSize})^\intercal\frac{\mathbf{a}'}{\protoCovSize}-\thresholdMatCin)}{\sqrt{p(1-p)(\frac{\mathbf{a}}{\protoCovSize})^\intercal\frac{\mathbf{a}'}{\protoCovSize}+q(1-q)(1-(\frac{\mathbf{a}}{\protoCovSize})^\intercal\frac{\mathbf{a}'}{\protoCovSize})}})|   & \\
        \qquad\qquad\qquad\qquad\leq \frac{0.7915}{\protoCovSize\sqrt{p(1-p)(\frac{\mathbf{a}}{\protoCovSize})^\intercal\frac{\mathbf{a}'}{\protoCovSize}+q(1-q)(1-(\frac{\mathbf{a}}{\protoCovSize})^\intercal\frac{\mathbf{a}'}{\protoCovSize})}} &
    \end{array}.
\end{align}
where $\Psi(x)=\int_{-\infty}^x \frac{1}{\sqrt{2\pi}} \exp(-x^2/2)$ is the cumulative distribution function for the normal distribution.
\end{lemma}
\begin{proof}
The proof can be found in Sec.~\ref{appendix:proof_lemma_SBM_genSyncMat_ULBounds} of the supplementary materials.
\end{proof}
By defining $k=h(\mathbf{a}),k'=h(\mathbf{a}')$, and $k''=h(\mathbf{a}'')$, we substitute \eqref{small_deviation_bound} into \eqref{CH_LB1} to further simplify the lower bound:
\begin{small}
\begin{align}\label{CH_LB2}
    \begin{array}{ll}
         \scaledCHdiv(\textrm{diag}(\mathbf{s})\genComConnect_{k},\textrm{diag}(\mathbf{s})\genComConnect_{k'}) & \\
         \geq  \displaystyle\max_{\mathbf{a}''\in \profileSet} 
         \frac{\mathbf{s}_{h(\mathbf{a}'')}}{2}
        \left(\sqrt{\max\left[
        \Psi(\frac{(\constintraComProb-\constextraComProb)((\frac{\mathbf{a}}{\protoCovSize})^\intercal\frac{\mathbf{a}''}{\protoCovSize}-\thresholdMatCin)}{\sqrt{(\constintraComProb-\constextraComProb)(\frac{\mathbf{a}}{\protoCovSize})^\intercal\frac{\mathbf{a}''}{\protoCovSize}+\constextraComProb}} \protoCovSize\sqrt{\scaleFineComProb})
        -
        \frac{0.7915}{\protoCovSize\sqrt{\scaleFineComProb}\sqrt{\constintraComProb(1-\intraComProb)(\frac{\mathbf{a}}{\protoCovSize})^\intercal\frac{\mathbf{a}''}{\protoCovSize}+\constextraComProb(1-\extraComProb)(1-(\frac{\mathbf{a}}{\protoCovSize})^\intercal\frac{\mathbf{a}''}{\protoCovSize})}},0\right]
        }
        \right.
         & \\
        \qquad \qquad \qquad \qquad \left.
        - \sqrt{\max\left[
        \Psi(\frac{(\constintraComProb-\constextraComProb)((\frac{\mathbf{a}'}{\protoCovSize})^\intercal\frac{\mathbf{a}''}{\protoCovSize}-\thresholdMatCin)}{\sqrt{(\constintraComProb-\constextraComProb)(\frac{\mathbf{a}'}{\protoCovSize})^\intercal\frac{\mathbf{a}''}{\protoCovSize}+\constextraComProb}} \protoCovSize\sqrt{\scaleFineComProb})
        +
        \frac{0.7915}{\protoCovSize\sqrt{\scaleFineComProb}\sqrt{\constintraComProb(1-\intraComProb)(\frac{\mathbf{a}'}{\protoCovSize})^\intercal\frac{\mathbf{a}''}{\protoCovSize}+\constextraComProb(1-\extraComProb)(1-(\frac{\mathbf{a}'}{\protoCovSize})^\intercal\frac{\mathbf{a}''}{\protoCovSize})}}
        ,0\right]}\right)^2   &     
    \end{array}.
\end{align}
\end{small}
\begin{assumption}\label{assumption:dominant_prior}
The prior vector $\mathbf{s}$ does not affect the dominant term in \eqref{CH_LB2}, i.e. the maximum term in \eqref{CH_LB2} is the same regardless of $\mathbf{s}_{h(\mathbf{a}'')}$ being included in the maximization or not. 
\end{assumption}
If Assumption \eqref{assumption:dominant_prior} holds, a simple  solution that would give an intuitively good estimate for the index of the dominant community in \eqref{CH_LB2} is:
\begin{align}\label{dominant_idx0}
\hat{k} = h(\hat{\mathbf{a}}), \quad \hat{\mathbf{a}}=
\left\lbrace
\begin{array}{ll}
\displaystyle\argmax_{\mathbf{a}''\in \profileSet} (\mathbf{a}-\mathbf{a}')^\intercal\mathbf{a}'' & \mathrm{if} \quad \constintraComProb>\constextraComProb\\
\displaystyle\argmax_{\mathbf{a}''\in \profileSet} (\mathbf{a}'-\mathbf{a}')^\intercal\mathbf{a}'' & \mathrm{if} \quad \constintraComProb<\constextraComProb
\end{array}\right..
\end{align}
It is easily seen that $\hat{\mathbf{a}}$ in both conditions of\eqref{dominant_idx0} is achieved for:
\begin{align}\label{dominant_idx}
\hat{\mathbf{a}}_u = \left\lbrace
\begin{array}{rl}
1 & \mathrm{if} \quad u=\displaystyle\argmax_{v}|\mathbf{a}_v-\mathbf{a}_v'|, \quad \constintraComProb\neq\constextraComProb\\
0 & \mathrm{else}
\end{array}\right..
\end{align}
Equation \eqref{dominant_idx} shows that the index of the estimated dominant term is independent of all the parameters, except for the index of the community pair it is calculated for, i.e. $\mathbf{a},\mathbf{a}'$ (or equivalently $k,k'$).

We continue simplifying \eqref{CH_LB2} by defining: 
\begin{align}\label{in_dominant_term_Def}
\omega \triangleq \displaystyle\max_{v}(\frac{\mathbf{a}_v}{\protoCovSize},\frac{\mathbf{a}_v'}{\protoCovSize}).
\end{align}
From \eqref{dominant_idx}, we have: 
\begin{align}\label{in_dominant_term_Def2}
\left\lbrace
\begin{array}{ll}
(\frac{\mathbf{a}}{\protoCovSize})^\intercal\frac{\hat{\mathbf{a}}}{\protoCovSize}=\omega, 
(\frac{\mathbf{a}'}{\protoCovSize})^\intercal\frac{\hat{\mathbf{a}}}{\protoCovSize}=0 & \mathrm{if} \quad \constintraComProb>\constextraComProb \quad \mathrm{or}\\
(\frac{\mathbf{a}}{\protoCovSize})^\intercal\frac{\hat{\mathbf{a}}}{\protoCovSize}=0, 
(\frac{\mathbf{a}'}{\protoCovSize})^\intercal\frac{\hat{\mathbf{a}}}{\protoCovSize}=\omega & \mathrm{if} \quad \constintraComProb<\constextraComProb
\end{array}\right..
\end{align}
Replacing \eqref{in_dominant_term_Def2} into \eqref{CH_LB2} yields:
\begin{align}\label{dominant_term}
\begin{array}{ll}
         \scaledCHdiv(\textrm{diag}(\mathbf{s})\genComConnect_{k},\textrm{diag}(\mathbf{s})\genComConnect_{k'}) \geq 
         \frac{\mathbf{s}_{\hat{k}}}{2}& \\
        \left\lbrace\begin{array}{ll}
        \left(\sqrt{\max\left[
        \Psi(\frac{(\constintraComProb-\constextraComProb)(\omega-\thresholdMatCin)}{\sqrt{(\constintraComProb-\constextraComProb)\omega+\constextraComProb}} \protoCovSize\sqrt{\scaleFineComProb})
        -
        \frac{0.7915}{\protoCovSize\sqrt{\scaleFineComProb}\sqrt{\constintraComProb(1-\intraComProb)\omega+\constextraComProb(1-\extraComProb)(1-\omega)}}
        ,0\right]} \right.
         & \mathrm{if} \quad \constintraComProb>\constextraComProb\\
        \qquad \qquad \qquad \left.
        - \sqrt{\max\left[
        \Psi(\frac{(\constintraComProb-\constextraComProb)(-\thresholdMatCin)}{\sqrt{\constextraComProb}} \protoCovSize\sqrt{\scaleFineComProb})
        +
        \frac{0.7915}{\protoCovSize\sqrt{\scaleFineComProb}\sqrt{\constextraComProb(1-\extraComProb)}}
        ,0\right]}\right)^2 & \\
        \left(\sqrt{\max\left[
        \Psi(\frac{(\constintraComProb-\constextraComProb)(-\thresholdMatCin)}{\sqrt{\constextraComProb}} \protoCovSize\sqrt{\scaleFineComProb})
        -
        \frac{0.7915}{\protoCovSize\sqrt{\scaleFineComProb}\sqrt{\constextraComProb(1-\extraComProb)}}
        ,0\right]} \right.
         & \mathrm{if} \quad \constintraComProb<\constextraComProb \\
        \qquad \left.
        - 
        \sqrt{\max\left[
        \Psi(\frac{(\constintraComProb-\constextraComProb)(\omega-\thresholdMatCin)}{\sqrt{(\constintraComProb-\constextraComProb)\omega+\constextraComProb}} \protoCovSize\sqrt{\scaleFineComProb})
        +
        \frac{0.7915}{\protoCovSize\sqrt{\scaleFineComProb}\sqrt{\constintraComProb(1-\intraComProb)\omega+\constextraComProb(1-\extraComProb)(1-\omega)}}
        ,0\right]}
        \right)^2 & 
        \end{array}\right.
         &
    \end{array}.
\end{align}
Note that even if Assumption \ref{assumption:dominant_prior} does not hold, or if \eqref{dominant_idx0} does not give the maximum term in \eqref{CH_LB2}, i.e. if \eqref{dominant_idx} is not the dominant index of \eqref{CH_LB2}, the inequality in \eqref{dominant_term} is still true as a lower bound.

For $0<\thresholdMatCin<\frac{1}{\COvar}$ and $\constintraComProb\scaleFineComProb, \constextraComProb\scaleFineComProb\leq \frac{1}{2}$, it is straightforward to see that the estimate dominant lower bound term in \eqref{dominant_term} increases as $\protoCovSize$, $\scaleFineComProb$, $\omega$, the gap between $\constintraComProb$ and $\constextraComProb$ increase, while other parameters remain unchanged. Increasing the gap between $\constintraComProb$ and $\constextraComProb$, is achieved by fixing whichever $\constintraComProb$ or $\constextraComProb$ that is smaller and increasing the other one. From the definition \eqref{in_dominant_term_Def}, $\omega$ only depends on the extended community index pair $k,k'$ for which it is calculated. As $\COvar$ increases, such community pair profiles, i.e. $\mathbf{a}$ and $\mathbf{a}'$, are allowed to cover more communities. Hence, from Def.~\ref{def:PS_def} and \ref{def:balanced}, a decrease in $\COvar$ results in an increase in $\omega$.
\end{proof}


\subsection{Proof of Corollary~\ref{corollary:nonsync_cond}}\label{appendix:proof_corollary_nonSyncCond}
\begin{proof}
In the following we address the recovery condition laid out in the corollary. 
The 
conditions are derived mainly using Abbe's community recovery conditions in { Theorem 1} of the general SBM paper \citep{abbe2015community}. 

From Lemma~\ref{lemma:SBM_genSyncMat}, we obtain that when $\matC$ is binarized to  $\biMatC$ according to \eqref{matC_binarized}, $\biMatC$ can be modeled as $\biMatC \sim \textrm{SBM}(\dimMatC, \mathbf{s}, \genComConnect)$, where $\genComConnect$ is defined in \eqref{U_poissonBinomial} and $\mathbf{s}$ is the length-$\nonSyncDimMatCom{\COvar}$ prior distribution vector for the c-nodes' extended community profile vector.
The probability of the $\text{MAP}$ estimator failure to recover $\genSyncMat$ from $\biMatC \sim \textrm{SBM}(\dimMatC, \mathbf{s}, \genComConnect)$ tends to zero, if for all $k,k'\in [\nonSyncDimMatCom{\COvar}]$ and $k\neq k'$ \citep{abbe2015community}:
\begin{align}\label{general_rec_condition}
    \left[\displaystyle\lim_{\dimMatC\xrightarrow[]{} \infty}  \frac{\dimMatC}{\log\dimMatC}\scaledCHdiv(\textrm{diag}(\mathbf{s})\genComConnect_{k},\textrm{diag}(\mathbf{s})\genComConnect_{k'})\right] \geq 1
\end{align}
where $\scaledCHdiv$ is the scaled CH divergence defined in \eqref{scaled_CH_divergence_def} (The term $\frac{\dimMatC}{\log\dimMatC}$ should be the coefficient of $\genComConnect$ in the original CH divergence definition, but we excluded it from $\scaledCHdiv$ for notation simplicity). $\scaledCHdiv$ is lower-bounded for the fix parameter $t=\frac{1}{2}$:
\begin{align}\label{scaled_CH_divergence_def_half}
    \begin{array}{ll}
         \scaledCHdiv(\textrm{diag}(\mathbf{s})\genComConnect_{k},\textrm{diag}(\mathbf{s})\genComConnect_{k'}) & \geq \displaystyle\sum_{k''\in[\nonSyncDimMatCom{\COvar}]}
         \mathbf{s}_{k''}[
        \frac{\genComConnect_{kk'}+\genComConnect_{k'k''}}{2} - \sqrt{\genComConnect_{kk'}\genComConnect_{k'k''}} 
         ]
    \end{array}.
\end{align}
From \eqref{scaled_CH_divergence_def_half}, we can easily deduce that if for all $k,k'\in [\nonSyncDimMatCom{\COvar}]$ and $k \neq k'$, the term
$\displaystyle\sum_{k''\in[\nonSyncDimMatCom{\COvar}]}  \mathbf{s}_{k''}[
        \frac{\genComConnect_{kk'}+\genComConnect_{k'k''}}{2} - \sqrt{\genComConnect_{kk'}\genComConnect_{k'k''}} 
         ]$ is strictly greater than zero, the condition \eqref{general_rec_condition} always satisfies (i.e. the LHS of \eqref{general_rec_condition} tends to $\infty$ as  $\dimMatC\xrightarrow[]{}\infty$). In order to have strictly positive inner summation terms, from the inequality of arithmetic and geometric means, for all $k,k'\in[\nonSyncDimMatCom{\COvar}]$, it is sufficient to have:
\begin{align}\label{cond1}
    \exists k''\in[\nonSyncDimMatCom{\COvar}] : \mathbf{s}_{k''}\left[
    \frac{\genComConnect_{k,k''}+\genComConnect_{k',k''}}{2} - \sqrt{\genComConnect_{k,k''}\genComConnect_{k',k''}}
    \right]>0.
\end{align}         
Equation \eqref{cond1} happens if for all $k,k'\in[\nonSyncDimMatCom{\COvar}]$
\begin{align}\label{cond3}
    \exists k''\in[\nonSyncDimMatCom{\COvar}] : \mathbf{s}_{k''}>0 , 
	|{\genComConnect_{k,k''}} - {\genComConnect_{k',k''}}| > 0 \
\end{align}

Continuing the condition derived in \eqref{cond3}, let $k=h(\mathbf{a}),k'=h(\mathbf{a}')$, and $k''=h(\mathbf{a}'')$ for the profile vectors $\mathbf{a},\mathbf{a}',\mathbf{a}''$ and the mapping function $h$. The assumptions of balanced and CO-$\COvar$ measurements in  Def.~\ref{def:PS_def} and \ref{def:balanced} suggest that for all profile vectors $\mathbf{a},\mathbf{a}' \in \profileSet$ where $\mathbf{a} \neq \mathbf{a}'$,  there exists at least a $u\in[\dimMatCom]$ where $\frac{\mathbf{a}_u}{\protoCovSize}\geq \frac{1}{\COvar},{\mathbf{a}_u'}=0$ or $\frac{\mathbf{a}_u'}{\protoCovSize}\geq \frac{1}{\COvar},{\mathbf{a}_u}=0$. 
 Without loss of generality, we assume the former to get ${\genComConnect_{k,k''}} - {\genComConnect_{k',k''}}>0$ (exchange the roles of $\mathbf{a}$ and $\mathbf{a}'$ to get ${\genComConnect_{k,k''}} - {\genComConnect_{k',k''}}<0$). 
We consider another profile vector $\mathbf{a}''\in \profileSet$ that satisfies $\frac{\mathbf{a}_u''}{\protoCovSize}=1$ and so, $(\frac{\mathbf{a}}{\protoCovSize})^\intercal\frac{\mathbf{a}''}{\protoCovSize} \geq \frac{1}{\COvar}$ and ${\mathbf{a}'}^\intercal\mathbf{a}''=0$
hold. 
Using the upper and lower bounds on $\genComConnect$ elements derived in Lemma \ref{lemma:SBM_genSyncMat_ULBounds}, we simplify the condition of \eqref{cond3} step-by-step:
\begin{align}\label{cond4}
\begin{array}{lll}
	|{\genComConnect_{k,k''}} - {\genComConnect_{k',k''}}|& \geq &
	|\Psi(\frac{(\intraComProb-\extraComProb)((\frac{\mathbf{a}}{\protoCovSize})^\intercal\frac{\mathbf{a}''}{\protoCovSize}-\thresholdMatCin)}{\sqrt{(\intraComProb-\extraComProb)(\frac{\mathbf{a}}{\protoCovSize})^\intercal\frac{\mathbf{a}''}{\protoCovSize}+\extraComProb}} \protoCovSize)
	-
\Psi(-\frac{(\intraComProb-\extraComProb)\thresholdMatCin}{\sqrt{\extraComProb}} \protoCovSize)|
 \\
& & -  \frac{0.7915}{\protoCovSize\sqrt{\intraComProb(1-\intraComProb)(\frac{\mathbf{a}}{\protoCovSize})^\intercal\frac{\mathbf{a}''}{\protoCovSize}+\extraComProb(1-\extraComProb)(1-(\frac{\mathbf{a}}{\protoCovSize})^\intercal\frac{\mathbf{a}''}{\protoCovSize})}}
- 
\frac{0.7915}{\protoCovSize\sqrt{\extraComProb(1-\extraComProb)}}
\\
& \geq &
|\Psi({\frac{(\intraComProb-\extraComProb)(\frac{1}{\COvar}-\thresholdMatCin)}{\sqrt{\frac{(\intraComProb-\extraComProb)}{\COvar}+\extraComProb}}} \protoCovSize) 
-
 \Psi(-{\frac{(\intraComProb-\extraComProb)\thresholdMatCin}{\sqrt{\extraComProb}}} \protoCovSize)|
 \\
& & 
-
\frac{1}{\protoCovSize} \left[
\frac{0.7915}{
\sqrt{\min(
\intraComProb(1-\intraComProb)(\frac{1}{\COvar})+\extraComProb(1-\extraComProb)(1-\frac{1}{\COvar})   ,    \intraComProb(1-\intraComProb) )}}
+\frac{0.7915}{\sqrt{\extraComProb(1-\extraComProb)}}
\right]
\\
& \geq &
|\Psi(\underbrace{\frac{(\constintraComProb-\constextraComProb)(\frac{1}{\COvar}-\thresholdMatCin)}{\sqrt{\frac{(\constintraComProb-\constextraComProb)}{\COvar}+\constextraComProb}}}_{\rho_1} \protoCovSize\sqrt{\scaleFineComProb}) 
-
 \Psi(-\underbrace{\frac{(\constintraComProb-\constextraComProb)\thresholdMatCin}{\sqrt{\constextraComProb}}}_{\rho_2} \protoCovSize\sqrt{\scaleFineComProb})|
 \\
& & 
-
\frac{1}{\protoCovSize\sqrt{\scaleFineComProb}} \left[
\frac{0.7915}{
\sqrt{\min(
\constintraComProb(1-\constintraComProb\scaleFineComProb)(\frac{1}{\COvar})+\constextraComProb(1-\constextraComProb\scaleFineComProb)(1-\frac{1}{\COvar})   ,    \constintraComProb(1-\constintraComProb\scaleFineComProb) )}}
+\frac{0.7915}{\sqrt{\constextraComProb(1-\constextraComProb\scaleFineComProb)}}
\right]
\\
&
> & 0 
\end{array}.
\end{align}
Assume a constant upper bound for $\scaleFineComProb\leq \upConstScaleFineComProb$ as $\dimMatF\geq \dimMatF_0$. Hence, we can upper bound the term:
\begin{align}
\begin{array}{ll}
\frac{0.7915}{
\sqrt{\min(
\constintraComProb(1-\constintraComProb\scaleFineComProb)(\frac{1}{\COvar})+\constextraComProb(1-\constextraComProb\scaleFineComProb)(1-\frac{1}{\COvar})   ,    \constintraComProb(1-\constintraComProb\scaleFineComProb) )}}
+\frac{0.7915}{\sqrt{\constextraComProb(1-\constextraComProb\scaleFineComProb)}}
& \\
 \leq 
\frac{0.7915}{
\sqrt{\min(
\constintraComProb(1-\constintraComProb\upConstScaleFineComProb)(\frac{1}{\COvar})+\constextraComProb(1-\constextraComProb\upConstScaleFineComProb)(1-\frac{1}{\COvar})   ,    \constintraComProb(1-\constintraComProb\upConstScaleFineComProb) )}}
+\frac{0.7915}{\sqrt{\constextraComProb(1-\constextraComProb\upConstScaleFineComProb)}} & 
\end{array}
\end{align}
To simplify the notations, we define:
\begin{align}\label{def_constants1}
\begin{array}{c}
 \sbmDegree\triangleq  \protoCovSize\sqrt{\scaleFineComProb}
,\qquad
 \rho_1 \triangleq \frac{(\constintraComProb-\constextraComProb)(\frac{1}{\COvar}-\thresholdMatCin)}{\sqrt{\frac{(\constintraComProb-\constextraComProb)}{\COvar}+\constextraComProb}}
,  \qquad
\rho_2 \triangleq \frac{(\constintraComProb-\constextraComProb)\thresholdMatCin}{\sqrt{\constextraComProb}} 
\\ 
\rho_3 \triangleq \left[
\frac{0.7915}{
\sqrt{\min(
\constintraComProb(1-\constintraComProb\scaleComProbAt{\dimMatF_0})(\frac{1}{\COvar})+\constextraComProb(1-\constextraComProb\scaleComProbAt{\dimMatF_0})(1-\frac{1}{\COvar})   ,    \constintraComProb(1-\constintraComProb\scaleFineComProb) )}}
+\frac{0.7915}{\sqrt{\constextraComProb(1-\constextraComProb\scaleComProbAt{\dimMatF_0})}}
\right]  
\end{array}.
\end{align}
To continue deriving a simpler formulation for the recovery condition, we consider two cases: $\rho_1,\rho_2>0$ or $\rho_1,\rho_2<0$. This way, we must have $\constintraComProb>\constextraComProb$ and $\constintraComProb<\constextraComProb$, respectively. So, considering both cases together: 
\begin{align}
\begin{array}{rl}
0 < \thresholdMatCin < \frac{1}{\COvar}, \quad \constintraComProb\neq\constextraComProb
\end{array}.
\end{align}
Substituting \eqref{def_constants1} and the equalities $\Psi(x)=\frac{1}{2}(1+\textrm{erf}(\frac{x}{\sqrt{2}}))$ and $\textrm{erf}(x)=-\textrm{erf}(-x)$ into \eqref{cond4}, with $\textrm{erf}(x)$ representing the Gauss Error Function, we get:
\begin{align}\label{cond5}
\begin{array}{rllllll}
|\Psi(\rho_1 \sbmDegree)-\Psi(-\rho_2 \sbmDegree)| & > & \frac{\rho_3}{\sbmDegree} & & & &
\\
|\frac{1}{2}\left[1+\textrm{erf}(\frac{\rho_1 \sbmDegree}{\sqrt{2}})\right] - \frac{1}{2}\left[1-\textrm{erf}(\frac{\rho_2 \sbmDegree}{\sqrt{2}})\right]| & > &  \frac{\rho_3}{\sbmDegree} & \mathrm{or} & |\frac{1}{2}\left[1-\textrm{erf}(\frac{-\rho_1 \sbmDegree}{\sqrt{2}})\right] - \frac{1}{2}\left[1+\textrm{erf}(\frac{-\rho_2 \sbmDegree}{\sqrt{2}})\right]| & > &  \frac{\rho_3}{\sbmDegree}\\
|\textrm{erf}(\frac{\rho_1 \sbmDegree}{\sqrt{2}}) + \textrm{erf}(\frac{\rho_2 \sbmDegree}{\sqrt{2}})| & > & \frac{2\rho_3}{\sbmDegree} & \mathrm{or} & |\textrm{erf}(\frac{-\rho_1 \sbmDegree}{\sqrt{2}}) + \textrm{erf}(\frac{-\rho_2 \sbmDegree}{\sqrt{2}})| & > & \frac{2\rho_3}{\sbmDegree}\\
|\textrm{erf}(\frac{|\rho_1| \sbmDegree}{\sqrt{2}}) + \textrm{erf}(\frac{|\rho_2| \sbmDegree}{\sqrt{2}})| & > & \frac{2\rho_3}{\sbmDegree}& & & &
\end{array}.
\end{align}
Next, we relax the condition in \eqref{cond4} using the inequality $\textrm{erf}(x)\geq 1-e^{-x^2}$ for all $x\geq 0$:
\begin{align}\label{cond6}
\begin{array}{rll}
1-e^{-{\frac{(|\rho_1| \sbmDegree)^2}{2}}} + 1-e^{-{\frac{(|\rho_2| \sbmDegree)^2}{2}}} & > & \frac{2\rho_3}{\sbmDegree} 
\\
e^{-{\frac{(\rho_1 \sbmDegree)^2}{2}}} + e^{-{\frac{(\rho_2 \sbmDegree)^2}{2}}} & < & 2-\frac{2\rho_3}{\sbmDegree} 
\end{array}.
\end{align}
We apply another relaxation to the condition in \eqref{cond6}, which yields:
\begin{align}\label{cond7}
\begin{array}{rll}
e^{-{\frac{(\max(|\rho_1|,|\rho_2|) \sbmDegree)^2}{2}}} & < & 1-\frac{\rho_3}{\sbmDegree} \\
-{\frac{(\max(|\rho_1|,|\rho_2|) \sbmDegree)^2}{2}} & < & \log(1-\frac{\rho_3}{\sbmDegree}) 
\end{array}.
\end{align}
The last relaxation is applied using $\log(x)\geq 1-\frac{1}{x}$:
\begin{align}\label{cond8}
\begin{array}{rll}
-{\frac{(\max(|\rho_1|,|\rho_2|) \sbmDegree)^2}{2}} & < & 1-\frac{1}{1-\frac{\rho_3}{\sbmDegree}} = - \frac{\rho_3}{\sbmDegree-\rho_3} \\
{\frac{(\max(|\rho_1|,|\rho_2|) \sbmDegree)^2}{2}} & > &  \frac{\rho_3}{\sbmDegree-\rho_3} \\
\sbmDegree^3 - \rho_3 \sbmDegree^2 - \underbrace{\frac{2\rho_3}{\max^2(|\rho_1|,|\rho_2|)}}_{\rho_4} & > & 0 
\end{array}.
\end{align}
The third-order polynomial in \eqref{cond8} has only one real root. Hence, the condition in \eqref{cond8} simplifies to the final form of:
\begin{align}\label{cond9}
\begin{array}{rl}
\protoCovSize\sqrt{\scaleFineComProb} > \Delta &
\end{array},
\end{align}
where
\begin{align}\label{def_constants2}
\begin{array}{c}
\Delta \triangleq \frac{1}{3}\left[  
\frac{\varpi}{\sqrt[3]{2}} + 
\frac{\sqrt[3]{2}\rho_3^2}{\varpi} + 
\rho_3
\right]
, \quad
 \rho_4 \triangleq \frac{2\rho_3}{\max^2(|\rho_1|,|\rho_2|)}, \\
 \varpi\triangleq \sqrt[3]{3\sqrt{3}\sqrt{4\rho_3^3\rho_4+27\rho_4^2}+2\rho_3^3+27\rho_4}
\end{array},
\end{align}
and the rest of the constants have already been defined in \eqref{def_constants1}. 

Next, we use the upper bound on the coverage size $\protoCovSize$ {(i.e. the number of measured fine nodes represented by a c-node)} from its definition in Def.~\ref{def_homogenous}, i.e. $\protoCovSize \leq \frac{\dimMatF}{\dimMatC}$, to further see which scaling functions $\scaleFineComProb$ is necessary to allow for the condition $\protoCovSize > \frac{\Delta}{\sqrt{\scaleFineComProb}}$ in \eqref{cond9}: 
\begin{align}\label{cond10}
\begin{array}{c}
    \frac{\dimMatF}{\dimMatC}>\frac{\Delta}{\sqrt{\scaleFineComProb}} 
\end{array}
\end{align}
Hence, the following is a necessary requirement to satisfy \eqref{cond10}: 
\begin{align}
    \scaleFineComProb>\Delta^2(\frac{\dimMatC}{\dimMatF})^2.
\end{align}
This completes the proof.
\end{proof}


\subsection{Proof of Theorem~\ref{thm:fullsync_UB_error_cond}}\label{appendix:proof_theorm_fullSyncUBCond}
\begin{proof}
We study the asymptotic behaviour of the $\text{MAP}$ failure error using \eqref{UB_MLFailureError} considering two scenarios of $\dimMatC$ being constant, or growing $\dimMatC\rightarrow\infty$. In the former case,
the error upper bound in \eqref{UB_MLFailureError} is the sum of finite terms, and $\mathbb{P}(\textrm{MAP failure}) \rightarrow 0$ 
if
\begin{align}\label{finiteL_zeroError_Condition}
    I_{} \rightarrow \infty.
\end{align}
In the latter scenario, i.e. when 
$\dimMatC\rightarrow\infty$
, it is proven in  { Theorem 3.2} of \citep{jog2015information} and { Theorem 5.1} of \citep{xu2020optimal} that the error goes to zero as:
\begin{align}\label{growingL_zeroError_Condition}
    \lim_{\dimMatC\rightarrow\infty} \frac{\dimMatC I_{}}{\dimMatCom \log{\dimMatC}}>1.
\end{align}
We first provide an upper bound for the $\log(.)$ term in \eqref{renyi_div_binomial}, using the inequality $\log(x)\leq x-1$ and the Taylor series expansion for $\frac{1}{\sqrt{1-x}} = 1+\frac{x}{2}+\frac{3}{8}x^2 + \cdots$:
\begin{align}\label{logterm0}
\begin{array}{ll}
     \log\left(\sqrt{(1-\intraComProb)(1-\extraComProb)}+\sqrt{\intraComProb\extraComProb} \right) & = \log(\sqrt{(1-\intraComProb)(1-\extraComProb)}) + \log(1+\frac{\sqrt{\intraComProb\extraComProb}}{\sqrt{(1-\intraComProb)(1-\extraComProb)}}) 
     \\
     & \leq \frac{1}{2}\log((1-\intraComProb)(1-\extraComProb)) + \frac{\sqrt{\intraComProb\extraComProb}}{\sqrt{(1-\intraComProb)(1-\extraComProb)}} 
     \\
     &  = \frac{1}{2}\log(1-\intraComProb-\extraComProb+\intraComProb\extraComProb) + \frac{\sqrt{\intraComProb\extraComProb}}{\sqrt{(1-\intraComProb)(1-\extraComProb)}}
     \\
     & \leq \frac{1}{2}(-\intraComProb-\extraComProb+\intraComProb\extraComProb) + \frac{\sqrt{\intraComProb\extraComProb}}{\sqrt{(1-\intraComProb)(1-\extraComProb)}}
     \\ 
     & = -\frac{\intraComProb+\extraComProb}{2} + \frac{\intraComProb\extraComProb}{2} + \sqrt{\intraComProb\extraComProb}(1+\frac{\intraComProb+\extraComProb-\intraComProb\extraComProb}{2}+\frac{3}{8}(\intraComProb+\extraComProb-\intraComProb\extraComProb)^2 + \cdots)
     \\
     & \leq -(\frac{\intraComProb+\extraComProb}{2}-\sqrt{\intraComProb\extraComProb}) + \frac{\intraComProb\extraComProb}{2} + \sqrt{\intraComProb\extraComProb}(\frac{\intraComProb+\extraComProb-\intraComProb\extraComProb}{2}+\frac{3}{8}(\intraComProb+\extraComProb-\intraComProb\extraComProb)^2 + \cdots)
\end{array}.
\end{align}
Next, we substitute \eqref{comConnect_constant_with_scaling} into \eqref{logterm0}:
\begin{align}\label{logterm1}
\begin{array}{ll}
     \log\left(\sqrt{(1-\intraComProb)(1-\extraComProb)}+\sqrt{\intraComProb\extraComProb} \right) & \leq -(\frac{\constintraComProb+\constextraComProb}{2}-\sqrt{\constintraComProb\constextraComProb})\scaleFineComProb + \frac{\constintraComProb\constextraComProb}{2}\scaleComProbAtPower{\dimMatF}{2} \\ & + \qquad\qquad\qquad \sqrt{\constintraComProb\constextraComProb}\scaleFineComProb(\frac{\constintraComProb+\constextraComProb-\constintraComProb\constextraComProb}{2}\scaleFineComProb+\frac{3}{8}(\constintraComProb+\constextraComProb-\constintraComProb\constextraComProb)^2\scaleComProbAtPower{\dimMatF}{2} + \cdots)
     \\
\end{array}.
\end{align}
As mentioned subsequent to the definition \eqref{comConnect_constant_with_scaling}, $\scaleFineComProb$ is a decreasing function with respect to its argument. Hence, we can rewrite \eqref{logterm1} for $\constintraComProb\neq\constextraComProb$, in terms of the dominant orders and denoting the \textit{equality} order by $\Theta$:
\begin{align}\label{logterm}
\begin{array}{ll}
     \log\left(\sqrt{(1-\intraComProb)(1-\extraComProb)}+\sqrt{\intraComProb\extraComProb} \right) & \leq -(\frac{\constintraComProb+\constextraComProb}{2}-\sqrt{\constintraComProb\constextraComProb})(\scaleFineComProb - \Theta(\scaleComProbAtPower{\dimMatF}{2}))
     \\
\end{array}.
\end{align}
Substituting \eqref{logterm} into \eqref{renyi_div_binomial}, \eqref{finiteL_zeroError_Condition}, and \eqref{growingL_zeroError_Condition},
yields the following conditions for $\mathbb{P}(\textrm{MAP failure}) \rightarrow 0$:
\begin{align}\label{raw_conditions0}
    \left\lbrace\begin{array}{ll}
        \left[
        (\frac{\constintraComProb+\constextraComProb}{2}-\sqrt{\constintraComProb\constextraComProb}) \protoCovSize^2 (\scaleFineComProb - \Theta(\scaleComProbAtPower{\dimMatF}{2}))\right] \rightarrow \infty & 
        \\
        \left[\displaystyle\lim_{\dimMatC\rightarrow\infty}
        \frac{2(\frac{\constintraComProb+\constextraComProb}{2}-\sqrt{\constintraComProb\constextraComProb})}{\dimMatCom} \protoCovSize^2(\scaleFineComProb - \Theta(\scaleComProbAtPower{\dimMatF}{2})) \frac{\dimMatC}{\log\dimMatC}\right]>1 & 
    \end{array}\right..
\end{align}
The first condition in \eqref{raw_conditions0} can only happen when $\protoCovSize\rightarrow\infty$ (since $\scaleFineComProb$ is a probability scaling and cannot approach $\infty$). From Def.~\ref{def_homogenous}, $\protoCovSize\rightarrow\infty$ only if $\dimMatF\rightarrow\infty$. This makes $(\frac{\constintraComProb+\constextraComProb}{2}-\sqrt{\constintraComProb\constextraComProb}) \protoCovSize^2 \scaleFineComProb(1 - \frac{\Theta(\scaleComProbAtPower{\dimMatF}{2})}{\scaleFineComProb}) \xrightarrow[]{\protoCovSize,\dimMatF\rightarrow\infty} \infty$ being equal to  $(\frac{\constintraComProb+\constextraComProb}{2}-\sqrt{\constintraComProb\constextraComProb}) \protoCovSize^2 \scaleFineComProb \xrightarrow[]{\protoCovSize,\dimMatF\rightarrow\infty} \infty$, since we assumed $\frac{\Theta(\scaleComProbAtPower{\dimMatF}{2})}{\scaleFineComProb}\xrightarrow[]{\dimMatF\rightarrow\infty} 0$. 

Similarly, the measurement size {(i.e. $\dimMatC$, the number of c-nodes}) can not exceed the fine graph size, i.e. $\dimMatC\leq\dimMatF$. Hence, by tending $\dimMatF\rightarrow\infty$, we simultaneously should have $\dimMatF\rightarrow\infty$. This way, the term $\displaystyle\lim_{\dimMatC\rightarrow\infty} (\scaleFineComProb - \Theta(\scaleComProbAtPower{\dimMatF}{2}))$ in the second condition in \eqref{raw_conditions0}, becomes $\displaystyle\lim_{\dimMatC,\dimMatF\rightarrow\infty} \scaleFineComProb $, since $\frac{\Theta(\scaleComProbAtPower{\dimMatF}{2})}{\scaleFineComProb}\rightarrow 0$. Accordingly, \eqref{raw_conditions0} simplifies to:
\begin{align}\label{raw_conditions}
    \left\lbrace\begin{array}{ll}
        \left[
        (\frac{\constintraComProb+\constextraComProb}{2}-\sqrt{\constintraComProb\constextraComProb}) \protoCovSize^2 \scaleFineComProb \right] \rightarrow \infty & 
        \\
        \left[\displaystyle\lim_{\dimMatC\rightarrow\infty}
        \frac{2(\frac{\constintraComProb+\constextraComProb}{2}-\sqrt{\constintraComProb\constextraComProb})}{\dimMatCom} \protoCovSize^2\scaleFineComProb \frac{\dimMatC}{\log\dimMatC}\right]>1 & 
    \end{array}\right..
\end{align}
Straightforward calculations summarize the recovery conditions in \eqref{raw_conditions} as \eqref{condition_exact_rec_regime}.
\end{proof}


\subsection{Proof of Lemma \ref{lemma:SBM_genSyncMat}}\label{appendix:binarized_SBM_dist}
\begin{proof}
Using the definition of $\biMatC$ in \eqref{matC_binarized}, the probability of community-wise connectivity can be calculated as:
\begin{align}\label{community_con_prob1}
    \begin{array}{rl}
         \genComConnect_{h(\genSyncVec_i),h(\genSyncVec_j)} & = \mathbb{P}[\biMatC_{ij}=1|h(\genSyncVec_i),h(\genSyncVec_j)]  \\
         & = \mathbb{P}[\matC_{ij}\geq \protoCovSize^2(\thresholdMatCin p + (1-\thresholdMatCin)q) |\genSyncVec_i,\genSyncVec_j] 
    \end{array}.
\end{align}
Considering the distribution of $\matC$ in \eqref{homogeneous_general_dist_matC}, the proof is complete.
\end{proof}


\subsection{Proof of Lemma \ref{lemma:fullsync_UB_error}}\label{appendix:proof_lemma_fullSyncUBError}
\begin{proof}
We define $I_{}$ as the Renyi divergence of order $\frac{1}{2}$ for discrete distributions $P$ and $Q$:
\begin{align}\label{renyi_def}
   I_{} \triangleq -2\log(\sum_{\ell\geq 0} \sqrt{P(\ell) Q(\ell) }).
\end{align}
The Renyi divergence $I_{}$ evaluates the extent to which $P$ and $Q$ are different from one another. $P$ and $Q$ are the intra- and extra-community edge weight distributions, which according to \eqref{matcDist_perfSync}, correspond respectively to $\textrm{Binomial}(\protoCovSize^2,\intraComProb)$ and $\textrm{Binomial}(\protoCovSize^2,\extraComProb)$ in this work. 
Straightforward calculations yields $I_{}$ in \eqref{renyi_div_binomial}.
\end{proof}


\subsection{Proof of Lemma \ref{lemma:SBM_genSyncMat_ULBounds}}\label{appendix:proof_lemma_SBM_genSyncMat_ULBounds}
\begin{proof}
From Lemma~\ref{lemma:SBM_genSyncMat}, we have:
\begin{align}
    \begin{array}{ll}
       \genComConnect_{k,k'}   & =
\mathbb{P}(X\geq \protoCovSize^2(\thresholdMatCin \intraComProb + (1-\thresholdMatCin)\extraComProb))
    \end{array},
\end{align}
where $ X \sim \textrm{PoissonBinomial}(\lbrace \intraComProb\rbrace^{\mathbf{a}^\intercal\mathbf{a}'}, \lbrace \extraComProb \rbrace^{\protoCovSize^2-\mathbf{a}^\intercal\mathbf{a}'})$. 
The Poisson binomial distribution can be approximated by the standard normal distribution with mean $\mu \triangleq \intraComProb\mathbf{a}^\intercal\mathbf{a}'+\extraComProb(\protoCovSize^2-\mathbf{a}^\intercal\mathbf{a}')$ and variance $\sigma^2 \triangleq \intraComProb(1-\intraComProb)\mathbf{a}^\intercal\mathbf{a}'+\extraComProb(1-\extraComProb)(\protoCovSize^2-\mathbf{a}^\intercal\mathbf{a}')$. Adopted from the Berry-Esseen theorem, such approximation comes with upper and lower  bounds formalizing the convergence rate. This way, the Poisson Binomial cumulative distribution function in \eqref{community_con_prob1} is approximated by the cumulative distribution function for the standard normal distribution denoted by $\Psi$, and upper and lower bounded by { Theorem 3.5} in \citep{tang2019poisson}:
\begin{align}\label{community_con_prob2}
    \begin{array}{ll}
         |\genComConnect_{h(\mathbf{a}),h(\mathbf{a}')} - \Psi(\frac{\mu-\protoCovSize^2(\thresholdMatCin \intraComProb + (1-\thresholdMatCin)\extraComProb)}{\sigma})|  & \leq \frac{0.7915}{\sigma}
    \end{array}
\end{align}
Substituting $\mu,\sigma$ into \eqref{community_con_prob2} gives:
\begin{align}\label{community_con_prob3}
    \begin{array}{ll}
         |\genComConnect_{h(\mathbf{a}),h(\mathbf{a}')} - \Psi(\frac{\intraComProb\mathbf{a}^\intercal\mathbf{a}'+\extraComProb(\protoCovSize^2-\mathbf{a}^\intercal\mathbf{a}')-\protoCovSize^2(\thresholdMatCin \intraComProb + (1-\thresholdMatCin)\extraComProb)}{\sqrt{\intraComProb(1-\intraComProb)\mathbf{a}^\intercal\mathbf{a}'+\extraComProb(1-\extraComProb)(\protoCovSize^2-\mathbf{a}^\intercal\mathbf{a}')}})|  & \leq \frac{0.7915}{\sqrt{\intraComProb(1-\intraComProb)\mathbf{a}^\intercal\mathbf{a}'+\extraComProb(1-\extraComProb)(\protoCovSize^2-\mathbf{a}^\intercal\mathbf{a}')}}
    \end{array}
\end{align}
Replacing $\intraComProb,\extraComProb$ from \eqref{comConnect_constant_with_scaling}, in addition to some straightforward calculations, completes the proof.
\end{proof}

\vfill

\end{document}